\documentclass[11pt]{article}
\usepackage{amssymb, amsmath}
\usepackage[mathscr]{eucal}
\usepackage{graphics,color}
\usepackage{hyperref}
\usepackage{bm}
 \usepackage{amsmath,amsthm, amsfonts,amssymb} 
    \usepackage[mathscr]{eucal}
    \usepackage{verbatim} 
    \usepackage{cases}
    \usepackage{graphicx}
    \usepackage{bm}
    \usepackage{float}
    \usepackage{url}
    \usepackage{setspace} 
    \usepackage[font=small,format=plain,labelfont=bf,up,textfont=it,up]{caption}

\newcommand{\overbar}[1]{\mkern 1.5mu\overline{\mkern-1.5mu#1\mkern-1.5mu}\mkern 1.5mu}

    \newtheorem{theorem}{Theorem}[section]
    \newtheorem{lemma}{Lemma}[section]

    \newtheorem{definition}{Definition}[section]
    \newtheorem{corollary}{Corollary}[section]

    \numberwithin{equation}{section}
    \newtheorem{remark}{Remark}[section]

    \numberwithin{equation}{section}
    \numberwithin{figure}{section}

    \newcommand{\Var}{\mathrm{Var}}

    \newcommand{\wh}{\widehat}

    \newcommand{\wt}{\widetilde}

    \newcommand{\newchapter}[3] 
	{                           
        \chapter[#2]{#3}
        \chaptermark{#1}
        \thispagestyle{myheadings}
	}
\allowdisplaybreaks
\makeatletter
\renewenvironment{titlepage}
 {%
  \if@twocolumn
    \@restonecoltrue\onecolumn
  \else
    \@restonecolfalse\newpage
  \fi
  \thispagestyle{empty}%
 }
 {%
  \if@restonecol
    \twocolumn
  \else
    \newpage
  \fi
 }
\makeatother

\begin{document}

\begin{titlepage}
\title{Extremes of locally stationary Gaussian and chi fields on manifolds\\[20pt]}
%
\author{Wanli Qiao \\
    Department of Statistics \\
George Mason University \\
4400 University Drive, MS 4A7 \\
Fairfax, VA 22030\\
USA\\
Email: \href{mailto:wqiao@gmu.edu}{wqiao@gmu.edu}\\}
\date{\today}
\maketitle
\vspace*{-0.8cm}

\begin{abstract} \noindent Depending on a parameter $h\in (0,1]$, let $\{X_h(\bm{t})$, $\bm{t}\in\mathcal{M}_h\}$ be a class of centered Gaussian fields indexed by compact manifolds $\mathcal{M}_h$. For locally stationary Gaussian fields $X_h$, we study the asymptotic excursion probabilities of $X_h$ on $\mathcal{M}_h$. Two cases are considered: (i) $h$ is fixed and (ii) $h\rightarrow0$. These results are extended to obtain the limit behaviors of the extremes of locally stationary $\chi$-fields on manifolds. 
\end{abstract}
\vfill 
{\small This research was partially support by the NSF-grant DMS 1821154
}

\end{titlepage}
\section{Introduction} \label{intro} 
We study the following two related problems in this manuscript. \\
 (i) Let $\{X(\bm{t})$, $\bm{t}\in\mathcal{M}\}$ be a centered Gaussian field indexed on a compact submanifold $\mathcal{M}$ of $\mathbb{R}^n$. We derive the asymptotic form of the excursion probability
 \begin{align}\label{resultintro0}
 \mathbb{P} \left(\sup_{\bm{t} \in\mathcal{M} }X(\bm{t}) > u\right),\; \text{as } u\rightarrow\infty.
 \end{align}
(ii) Let $ \left\{X_h(\bm{t}), \bm{t} \in \mathcal {M}_h\right\}_{h\in(0,1]}$ be a class of centered Gaussian fields, where $\mathcal {M}_h$ are compact submanifolds of ${\mathbb R}^n$. Suppose that we have the structure $\mathcal {M}_h=\mathcal {M}_{h,1} \times \mathcal {M}_{h,2}$ such that $\bm{t}=(\bm{t}_{(1)}^T, \bm{t}_{(2)}^T)^T\in\mathcal {M}_h$ means $\bm{t}_{(1)}\in\mathcal {M}_{h,1}$ and $\bm{t}_{(2)}\in\mathcal {M}_{h,2}$, where we allow $\mathcal {M}_{h,2}$ to be a null set. The Gaussian fields $X_h(\bm{t})$ we consider has a rescaled form $X_h(\bm{t}) = \overbar X_h(\bm{t}_{(1)}/h,\bm{t}_{(2)}), \bm{t} \in \mathcal {M}_h$ for some $\overbar X_h$ satisfying a local stationarity condition. We derive the following limit result 
\begin{align}\label{resultintro}
\lim_{h\rightarrow 0}\mathbb{P}\left(a_h\left(\sup_{\bm{t}\in\mathcal{M}_h} X_h(\bm{t}) - b_h\right)\leq z\right) = e^{-e^{-z}},
\end{align}
for some $a_h, b_h \in {\mathbb R}_+$ and fixed $z \in {\mathbb R}$. \\

While there is a large amount of literature on excursion probabilities of Gaussian processes or fields (see, e.g., Adler and Taylor \cite{Adler:2007}, and Aza\"{i}s and Wschebor \cite{Azais:2009}), most of the existing work only considers index sets $\mathcal{M}$ (or $\mathcal{M}_h$) of dimension $n$ (the same as the ambient Euclidean space), while we focus on Gaussian fields indexed by manifolds that can be low-dimensional. \\

For problem (i), some relevant results can be found in Mikhaleva and Piterbarg {\cite{Mikhaleva:1997}, Piterbarg and Stamatovich \cite{Piterbarg:2001}, and Cheng \cite{Cheng:2017}. Compared with these works, the framework of our result is more general in the following aspects: First of all, Cheng \cite{Cheng:2017} studies the excursion probabilities of {\em locally isotropic} Gaussian random fields on manifolds, where local isotropicity means the variance between two local points only depends on their (geodesic) distance, while we consider {\em locally stationary} Gaussian fields, for which not only the distance between the points but also their locations are involved in the variance. Furthermore, in Mikhaleva and Piterbarg {\cite{Mikhaleva:1997} and Piterbarg and Stamatovich \cite{Piterbarg:2001}, the Gaussian fields are assumed to be indexed by $\mathbb{R}^n$, while we only require the index sets to be the manifolds. As pointed out in Cheng \cite{Cheng:2017}, it is not clear whether one can always find a Gaussian field indexed by $\mathbb{R}^n$ whose restriction on $\mathcal{M}$ is $X(\bm{t})$. Also see Cheng and Xiao \cite{Cheng:2016} for some further arguments on this point. In addition, all the above works assume that the manifolds are smooth ($C^\infty$), while we consider a much larger class of manifolds (only satisfying a {\em positive reach} condition). In fact, the properties of positive reach play a critical role in the geometric construction in our proofs. \\ 

For problem (ii), the study in Qiao and Polonik \cite{Qiao:2018} corresponds to a special case of (\ref{resultintro}) when $\mathcal{M}_h\equiv\mathcal{M}$ for some manifold $\mathcal{M}$ independent of $h$, and $\mathcal{M}_{h,2}=\emptyset$. They use some ideas from Mikhaleva and Piterbarg {\cite{Mikhaleva:1997}  and also assume that $X_h$ is indexed by a neighborhood of higher dimensions around $\mathcal{M}$, while we only need $X_h$ to be indexed by the manifolds $\mathcal{M}_h$. This weaker requirement for the Gaussian fields finds broader applications when the Gaussian fields are observable or can be approximated only on low-dimensional manifolds. See (\ref{crfunction}) below for example. Also, by using the assumed structure of $\mathcal{M}_h$, only rescaling the parameters $\bm{t}_1$ allows us to apply (\ref{resultintro}) to get asymptotic extreme value distributions of $\chi$-fields on manifolds, which in fact is one of the motivations of this work, as described below. \\

Let $ \{\bm{X}(\bm{s}), \bm{s} \in \mathcal{M}\}$ be a $p$-dimensional Gaussian vector field, where $\bm{X}=(X_{1},\cdots,X_{p})^T$ has zero mean and identity variance-covariance matrix. Note that we have suppressed the possible dependence of $\bm{X}$ and $\mathcal{M}$ on $h$. Define
\begin{align}\label{chifielddef}
\chi(\bm{s})= [X_{1}^2(\bm{s})+\cdots+X_{p}^2(\bm{s})]^{1/2} , \; \bm{s} \in \mathcal{M}, 
\end{align}
which is called a $\chi$-field, where we allow the components $X_{i}(\bm{s}_i)$ and $X_{j}(\bm{s}_j)$ to be dependent, if $\bm{s}_i\neq \bm{s}_j$. Let $\mathbb{S}^{p-1}=\{x\in\mathbb{R}^p:\; \|x\|=1\}$ be the unit $(p-1)$-sphere. Using the property of Euclidean norm, we have 
\begin{align}\label{relation}
\sup_{\bm{s}\in\mathcal{M}}\chi(\bm{s}) = \sup_{\bm{s}\in\mathcal{M}, {\bm{v}}\in \mathbb{S}^{p-1}} Y_h(\bm{s}, {\bm{v}}), 
\end{align}
where ${\bm{v}}=(v_1,\cdots,v_p)\in\mathbb{R}^p$ and
\begin{align*}
Y(\bm{s}, {\bm{v}}) = X_{1}(\bm{s})v_1 + \cdots + X_{p}(\bm{s})v_p, \; \bm{s}\times {\bm{v}}\in\mathcal{M} \times \mathbb{S}^{p-1} .
\end{align*}
Note that $Y(\bm{s}, {\bm{v}})$ is a zero-mean and unit-variance Gaussian field on $\mathcal{M}\times \mathbb{S}^{p-1}$. Using the relation in (\ref{relation}) and by applying the results in (\ref{resultintro0}) and (\ref{resultintro}), we can study the asymptotic excursion probabilities of $\sup_{\bm{s}\in\mathcal{M}}\chi(\bm{s})$ as well as obtain a result in the form of
%
\begin{align}\label{resultintrochi}
\lim_{h\rightarrow 0}\mathbb{P}\left(a_h\left(\sup_{\bm{s}\in\mathcal{M}} \chi(\bm{s}/h) - b_h\right)\leq z\right) = e^{-e^{-z}}.
\end{align}

The result in (\ref{resultintrochi}) has the following two interesting applications. We consider a vector-valued signal plus noise model
\begin{align}
\wh{{\bm{f}}}_h(\bm{s}) = {\bm{f}}(\bm{s}) + \bm{X}(\bm{s}/h), \; \bm{s}\in\mathcal{M},
\end{align}
where ${\bm{f}}(\bm{s})$ is a $p$-dimensional signal, $\bm{X}(\bm{s})$ is the noise modeled by the Gaussian vector field considered above. We assume that only $\wh{{\bm{f}}}_h(\bm{s})$ is directly observable. Given $\alpha\in(0,1)$, let $z_\alpha$ be such that $\exp(-\exp(-z_\alpha))=1-\alpha$.\\

(a) Suppose that $\mathcal{M}$ is known, and the inference for the signal ${\bm{f}}(\bm{s})$ is of interest. We have the following asymptotic $(1-\alpha)$ confidence tube for ${\bm{f}}(\bm{s})$:
\begin{align}\label{crfunction}
\mathcal{G}_h(\bm{s}) := \left\{{\bm{g}}\in\mathbb{R}^p:\; a_h\left(\|\wh{{\bm{f}}}_h(\bm{s}) - {\bm{g}} \| - b_h\right)\leq z_\alpha\right\}, \; \bm{s}\in\mathcal{M}.
\end{align}
In other words, $\mathbb{P}({\bm{f}}(\bm{s}) \in \mathcal{G}_h(\bm{s}),\; \forall \bm{s}\in\mathcal{M}) \rightarrow 1-\alpha$, as $h\rightarrow0$.\\

(b) Suppose that the manifold $\mathcal{M}$ is unknown but implicitly defined by $\mathcal{M}=\{\bm{s}\in\mathcal{A}: {\bm{f}}(\bm{s})={\bm{g}}_0\}$, where $\mathcal{A}\subset\mathbb{R}^n$ is a known neighborhood of $\mathcal{M}$ (say, a unit cube), and ${\bm{g}}_0$ is a known $p$-dimensional vector so that $\mathcal{M}$ is the intersection of multiple level sets. Suppose that $\wh{{\bm{f}}}_h(\bm{s})$ is observable on $\mathcal{A}$, and the inference for the manifold $\mathcal{M}$ is of interest. We have the following asymptotic $(1-\alpha)$ confidence region for $\mathcal{\mathcal{M}}$:
\begin{align}\label{crmanifold}
\mathcal{F}_h := \left\{\bm{s}\in \mathcal{A}:\; a_h\left( \|\wh{{\bm{f}}}_h(\bm{s}) - {\bm{g}}_0 \| - b_h\right)\leq z_\alpha\right\}.
\end{align}
That is, $\mathbb{P}(\mathcal{M}\subset \mathcal{F}_h) \rightarrow 1-\alpha$, as $h\rightarrow0$.\\

In statistics the suprema of empirical processes can be approximated by the suprema of Gaussian processes or fields under regularity assumptions (see Chernozhukov et al. \cite{Chernozhukov:2014}). Applying results in (a) and (b) to the approximating Gaussian fields, one can study the statistical inference for a large class of objects including functions and geometric features (low-dimensional manifolds). In a form similar to (\ref{crfunction}), confidence bands for density functions are given in Bickel and Rosenblatt \cite{Bickel:1973a} and Rosenblatt \cite{Rosenblatt:1976}. Similar work for regression functions can be found in Konakov and Piterbarg \cite{Konakov:1984}. We note that in these examples the study of the suprema of the approximating Gaussian processes or fields focuses on $\mathcal{M}$ being compact intervals or hypercubes. We expect that our result (\ref{crfunction}) is useful in studying functions supported on more general (low-dimensional) manifolds, especially in the context of {\em manifold learning}, which usually assumes that data lie on low-dimensional manifolds embedded in high-dimensional space. The result (\ref{crmanifold}) is useful to infer the location of the manifolds. In fact, the results proved in this work provide the probabilistic foundation to our companion work Qiao \cite{Qiao:2019}, where the confidence regions for density ridges are obtained. Ridges are low-dimensional geometric features (manifolds) that generalize the concepts of local modes, and have been applied to model filamentary structures such as the Cosmic Web and road systems. See Qiao and Polonik \cite{Qiao:2019b} for a similar application for the construction of confidence regions for level sets.\\

The study of the asymptotic extreme value behaviors of $\chi$-processes and fields has drawn quite some interest recently. To our best knowledge, the study in the existing literature has only focused on $\chi$-processes and fields indexed by intervals or hyper cubes, but not low-dimensional manifolds. See, for example,  Albin et al. \cite{Albin:2016}, Bai~\cite{Bai:2018}, Hashorva and Ji \cite{Hashorva:2015}, Ji et al. \cite{Ji:2019}, Konstantinides et al. \cite{Konstantinides:2004}, Lindgren \cite{Lindgren:1989}, Ling and Tan \cite{Ling:2016}, Liu and Ji \cite{Liu:2016, Liu:2017}, Piterbarg \cite{Piterbarg:1994, Piterbarg:1996}, Tan and Hashorva \cite{Tan:2013a, Tan:2013b}, Tan and Wu \cite{Tan:2014}. Also it is worth mentioning that it is often assumed that $X_1,\cdots,X_r$ are independent copies of a Gaussian process or field $X$ in the literature, while the cross-dependence among $X_1,\cdots,X_r$ is allowed under certain constraints in this work. The cross-dependence structures of multivariate random fields have been important objects to study in multivariate geostatisitics (see Genton and Kleiber~\cite{Genton:2015}).\\

The manuscript is organized as follows. In Section~\ref{unrescaledM} we introduce the concepts that we use in this work to characterize the manifolds (positive reach) and the Gaussian fields (local stationarity). Then the result for (\ref{resultintro0}) (called the unscaled case) is formulated in Theorem~\ref{fixedmanifold}, As an application, a similar result for the $\chi$-fields in presented in Corollary~\ref{chifixedmanifold}. In Section~\ref{MainResult} we give the result (\ref{resultintro}) (called the rescaled case) in Theorem~\ref{ProbMain} and its $\chi$-fields extension in Corollary~\ref{rescaledchifield}. All the proofs are presented in Section~\ref{proof}, and Section~\ref{appendix} contains some miscellaneous results used in the manuscript. 

\section{Extremes of unscaled Gaussian and $\chi$ fields on manifolds}\label{unrescaledM}
We consider a centered Gaussian field $X(\bm{t}), \; \bm{t}\in\mathcal{M}$, where $\mathcal{M}$ is a $r$-dimensional submanifold of $\mathbb{R}^n$ ($1\leq r\leq n$). Let $r_X(\bm{t}_1,\bm{t}_2)=\text{Cov}(X(\bm{t}_1),X(\bm{t}_2))$ for any $\bm{t}_1,\bm{t}_2\in\mathcal{M}$. We first introduce some concepts we need to characterize the covariance $r_X$ of the Gaussian field $X$ and the manifold $\mathcal{M}$.\\

For a positive integer $k\leq n$, let $E=\{e_1,\cdots,e_k\}$ be a collection of positive integers such that $n=e_1+\cdots +e_k$, and let $\pmb{\alpha}=\{\alpha_1,\cdots,\alpha_k\}$ be a collection of positive numbers. Then the pair $(E,\pmb{\alpha})$ is called a structure. Let $\|\cdot\|$ denote the Euclidean norm. Denote $E(0)=0$ and $E(i)=e_1+\cdots+e_i$, $i=1,\cdots,k$. For any $\bm{t}=(t_1,\cdots,t_n)^T\in\mathbb{R}^n$, its structure module is denoted by $|\bm{t}|_{E,\alpha} = \sum_{i=1}^k \|\bm{t}_{(i)}\|^{\alpha_i},$ where $\bm{t}_{(i)}=(t_{E(i-1)+1},\cdots,t_{E(i)})^T$.\\

Suppose that $\alpha_i\leq 2, i=1,\cdots,k,$ and consider  a Gaussian field $W(\bm{t}), \bm{t}\in\mathbb{R}^n,$ with continuous trajectories such that $\mathbb{E} W(\bm{t}) = -|\bm{t}|_{E,\pmb{\alpha}}$ and $\text{Cov}(W(\bm{t}),W(\bm{s}))=|\bm{t}|_{E,\pmb{\alpha}}+|\bm{s}|_{E,\pmb{\alpha}}-|\bm{t}-\bm{s}|_{E,\pmb{\alpha}}$. It is known that such a field exists (see page 98, Piterbarg~\cite{Piterbarg:1996}). 
For any measurable subset $\mathcal{T}\subset\mathbb{R}^n$ define
\begin{align*}
H_{E,\pmb{\alpha}}(\mathcal{T})=\mathbb{E}\exp\Big(\sup_{\bm{t}\in \mathcal{T}}W(\bm{t})\Big).
\end{align*}

For any $T>0$, denote $[0,T]^n = \{\bm{t}\in\mathbb{R}^n: t_i\in [0,T]\}$. The generalized Pickands' constant is defined as 
\begin{align*}
H_{E,\pmb{\alpha}} = \lim_{T\rightarrow\infty} \frac{H_{E,\pmb{\alpha}}([0,T]^n)}{T^n},
\end{align*}
which is a positive finite number. When $k=1$, $E=\{1\}$ and $\pmb{\alpha}=\alpha\in(0,2]$, we denote $H_{E,\pmb{\alpha}}=H_\alpha$.

\begin{definition}[local-$(E, \pmb{\alpha},D_{\bm{t}})$-stationarity]
Let $ \{Z(\bm{t}), \bm{t} \in \mathcal{M} \}$ be a Gaussian random field with covariance function $r_Z$, indexed on a submanifold $\mathcal{M}$ of ${\mathbb R}^n$. $Z$ is said to be locally-$(E,\pmb{\alpha},D_{\bm{t}})$-stationary on $\mathcal{M}$, if for all $\bm{t}\in \mathcal{M}$ there exists a nonsingular matrix $D_{\bm{t}}$ such that
\begin{align}\label{covariancecondition}
r_Z(\bm{t}_1,\bm{t}_2)= 1 - | D_{\bm{t}} (\bm{t}_1-\bm{t}_2)|_{E,\pmb{\alpha}} (1+o(1)),
\end{align}
as $\max\{\|\bm{t}-\bm{t}_1\|,\|\bm{t}-\bm{t}_2\|\}\rightarrow0$ for $\bm{t}_1,\bm{t}_2\in \mathcal{M}$.
\end{definition}

{\em Positive reach}: We use the concept of reach to characterize the manifold $\mathcal{M}$. For a set $A\subset\mathbb{R}^n$ and a point $x\in\mathbb{R}^n$, let $d(x,A)=\inf \{\|x-y\|:\; y\in A\}$ be the distance from $x$ to $A$. The normal projection onto $A$ is defined as $\pi_A(x)=\{y\in A: \|x-y\|=d(x,A)\}$. For $\delta>0$, let $\mathcal{B}(x,\delta)=\{y\in\mathbb{R}^n: \|x-y\|\leq \delta\}$ be the ball centered at $x$ with radius $\delta$. The reach of $A$, denoted by $\Delta(A)$, is defined as the largest $\delta>0$ such that for each point $x\in \cup_{y\in A}\mathcal{B}(y,\delta)$, $\pi_A(x)$ consists of a single point. See Federer \cite{Federer:1959}. The reach of a manifold is also called condition number (see Niyogi et al. \cite{Niyogi:2008}). A closed submanifold of $\mathbb{R}^n$ has positive reach if and only if it is $C^{1,1}$ (see Scholtes, \cite{Scholtes:2013}). Here a $C^{1,1}$ manifold by definition is a $C^1$ manifold equipped with a class of atlases whose transition maps have Lipschitz continuous first derivatives. The concept of positive reach is also closely related to ``r-convexity'' and ``rolling conditions'' (Cuevas et al. \cite{Cuevas:2012}).\\ 

Suppose that the structure $(E,\pmb{\alpha})$ is given. Let $R=\{r_1,\cdots,r_k\}$ be a collection of positive integers such that $r_i\leq e_i$, $i=1,\dots,k$, for which we denote $R\leq E$. Let $r=r_1+\cdots+r_k$. We impose the following assumptions on the manifold $\mathcal{M}$ and the Gaussian field $X(\bm{t}), \; \bm{t}\in\mathcal{M}$:
\begin{itemize}
\item[(A1)] For $R\leq E$, we assume that $\mathcal{M} = \mathcal{M}_1\times\cdots \times \mathcal{M}_k$, where for $i=1,\cdots,k$, $\mathcal{M}_i$ is a $r_i$-dimensional compact submanifold of $\mathbb{R}^{e_i}$ with positive reach and positive $r_i$-dimensional Lebesgue measure.
\item[(A2)] Let $D_{\bm{t}}=\text{diag}(D_{1,\bm{t}},\cdots,D_{k,\bm{t}})$ be a block diagonal matrix, where the dimension of $D_{i,\bm{t}}$ is $e_i\times e_i$, and the matrix-valued function $D_{i,\bm{t}}$ is continuous in $\bm{t}\in\mathcal{M}$, for $i=1,\cdots,k$. For $0<\alpha_1,\cdots,\alpha_k\leq 2$, we assume that the Gaussian field $X(\bm{t})$ on $\mathcal{M}$ has zero mean and is locally-$(E,\pmb{\alpha},D_{\bm{t}})$-stationary. 
\end{itemize}
\begin{remark}
Note that the local stationarity condition for the Gaussian field is given using the structure $(E,\pmb{\alpha})$ for $\mathbb{R}^n$. The structural assumptions on $\mathcal{M}$ and $D_{\bm{t}}$ in (A1) and (A2) are used to guarantee that a similar structure $(R,\pmb{\alpha})$ can be found when the local stationarity of the Gaussian field is expressed on a low-dimensional manifold, which locally resembles $\mathbb{R}^r$. Note that, however, in the special case of $k=1$ we do not have these structural constraints for $\mathcal{M}$ and $D_{\bm{t}}$ any more. 
\end{remark}

{\em Some notation:} Let $1\leq m \leq n$. For an  $n \times m$ matrix $G$, let $\|G\|^2_m$ be the sum of squares of all minor determinants of order $m$. Let $\mathcal{H}_m$ denote the $m$-dimensional volume measure. For a $C^1$ manifold $M$, at each $u\in M,$ let $T_u M$ denote the tangent space of $M$ at $u$. Let $\phi$ and $\Phi$ denote the standard normal density and cumulative distribution function, respectively, and let  $\bar\Phi(u)=1-\Phi(u)$ and $\Psi(u)=u^{-1}\phi(u)$. Recall that $\bm{t}=(\bm{t}_{(1)}^T,\cdots,\bm{t}_{(k)}^T)^T$. The following is a result for the asymptotic behavior of the excursion probability of $X$ on the manifold $\mathcal{M}$.

\begin{theorem}\label{fixedmanifold}
For a Gaussian field $X(\bm{t})$, $\bm{t}\in\mathcal{M}$ satisfying assumptions (A1) and (A2), if $r_X(\bm{t},\bm{s})<1$ for all $\bm{t},\bm{s}$ from $\mathcal{M}$, $\bm{t}\neq\bm{s}$, then
\begin{align}
\mathbb{P} \left(\sup_{\bm{t}\in \mathcal{M}} X(\bm{t})> u\right) = H_{R,\alpha} \int_{\mathcal{M}} \prod_{j=1}^k \|D_{j,\bm{t}} P_{j,\bm{t}_{(j)}}\|_{r_j} d\mathcal{H}_r(\bm{t})  \prod_{i=1}^k u^{2r_i/\alpha_i} \Psi(u) (1+o(1)),
\end{align}
as $u\rightarrow\infty,$ where $P_{j,\bm{t}_{(j)}}$ is an $e_j\times r_j$ matrix whose columns are orthonormal and span the tangent space of $T_{\bm{t}_{(j)}}\mathcal{M}_j$. 
\end{theorem}

\begin{remark}\label{factorization}
The factorization lemma (Lemma 6.4, Piterbarg~\cite{Piterbarg:1996}) implies that $H_{R,\alpha} = \prod_{i=1}^ k H_{r_i,\alpha_i},$ where in the notation we do not distinguish between $r_i$ (or $\alpha_i$) and $\{r_i\}$ (or $\{\alpha_i\}$).
\end{remark}

We will apply the above theorem to study the excursion probabilities of $\chi$-fields indexed by manifolds. Let $ \{\bm{X}(\bm{s}), \bm{s} \in \mathcal{L}\}$ be a centered $p$-dimensional ($p\geq2$) Gaussian vector field, where $\bm{X}=(X_{1},\cdots,X_{p})^T$ with $\text{Var}(X_i)=1$, $i=1,\cdots,p$, and $\mathcal{L}$ is a $m$-dimensional submanifold of ${\mathbb R}^n$ ($1\leq m\leq n$). We consider the asymptotics of 
\begin{align}\label{chiextreme}
\mathbb{P}\left(\sup_{\bm{s}\in\mathcal{L}}\|\bm{X}(\bm{s})\| > u\right),\; \text{ as } u\rightarrow\infty.
\end{align}
Let ${\bm{v}}=(v_1,\cdots,v_p)^T\in\mathbb{R}^p$, $\bm{t}=(\bm{s}^T, {\bm{v}}^T)^T\in\mathbb{R}^{n+p}$, and 
\begin{align}\label{Yexpress}
Y(\bm{t})=Y(\bm{s},{\bm{v}})=X_1(\bm{s})v_1+\cdots+X_p(\bm{s})v_p.
\end{align}
Due to the relation in (\ref{relation}), it is clear that (\ref{chiextreme}) is equivalent to
\begin{align}\label{chiextremealt}
\mathbb{P}\left(\sup_{\bm{t}\in\mathcal{L}\times \mathbb{S}^{p-1}} Y(\bm{t}) > u\right),\; \text{ as } u\rightarrow\infty.
\end{align} 

To study (\ref{chiextreme}) through (\ref{chiextremealt}), we directly impose an assumption on the covariance function $r_Y$ of $Y$, which we find convenient because it allows us to encode the possible cross-dependence structure among $X_{1},\cdots,X_{r}$ into $r_Y$. See example (ii) below. For $i=1,2$, denote $\bm{t}_i=(\bm{s}_i^T, {\bm{v}}_i^T)^T$, where ${\bm{v}}_i^T=(v_{i,1},\cdots,v_{i,p})$. Let $r_Y(\bm{t}_1,\bm{t}_2)=\text{Cov}(Y(\bm{t}_1),Y(\bm{t}_2))$. Then notice that 
\begin{align}\label{covstructure}
r_Y(\bm{t}_1,\bm{t}_2) = &\sum_{i=1}^p \sum_{j=1}^p \text{Cov}(X_{i}(\bm{s}_1),X_{j}(\bm{s}_2))v_{1,i}v_{2,j} \nonumber\\
= & {\bm{v}}_1^T {\bm{v}}_2 - \sum_{i=1}^p \sum_{j=1}^p [\delta_{ij} - \text{Cov}(X_{i}(\bm{s}_1),X_{j}(\bm{s}_2)) ] v_{1,i}v_{2,j} \nonumber \\
= & 1 - \frac{1}{2}\|{\bm{v}}_1-{\bm{v}}_2\|^2 - \sum_{i=1}^p \sum_{j=1}^p [\delta_{ij} - \text{Cov}(X_{i}(\bm{s}_1),X_{j}(\bm{s}_2)) ] v_{1,i}v_{2,j},
\end{align}
where $\delta_{ij}=\mathbf{1}(i=j)$ is the Kronecker delta. The structure in (\ref{covstructure}) suggests the following assumption on $r_Y(\bm{t}_1,\bm{t}_2)$.
\begin{itemize}
\item[(A3)] We assume that $Y(\bm{t})$ given in (\ref{Yexpress}) is a local-$(E, \pmb{\alpha}, D_{\bm{t}})$-stationary Gaussian field on $\mathcal{L}\times \mathbb{S}^{p-1}$ with $D_{\bm{t}} = \text{diag}(B_{\bm{t}} ,\frac{1}{\sqrt{2}}{\bm{I}}_{p})$, where $B_{\bm{t}}$ is a nonsingular $n\times n$ dimensional matrix for all $\bm{t}\in \mathcal{L}\times \mathbb{S}^{p-1}$, $E=\{n,p\}$ and $\pmb{\alpha} = \{\alpha,2\},$  for some $0<\alpha\leq 2.$ We assume that matrix-valued function $B_{\bm{t}}$ is continuous in $\bm{t}\in \mathcal{L}\times \mathbb{S}^{p-1}$.
\end{itemize}

\begin{remark}
Note that assumption (A3) implies that for $\bm{s}\in\mathcal{L}$ and $1\leq i, j\leq p$
\begin{align*}
\text{Cov}(X_i(\bm{s}),X_j(\bm{s}))=
\begin{cases}
0 & i\neq j\\
1 & i=j
\end{cases}.
\end{align*}
In other words, we are considering a Gaussian vector field $\bm{X}(\bm{s})$ whose variance-covariance matrix at any point $\bm{s}\in\mathcal{L}$ has been standardized. However, cross-dependence between $X_i(\bm{s}_i)$ and $X_j(\bm{s}_j)$ is still possible under assumption (A3) for $\bm{s}_i,\bm{s}_j\in\mathcal{L}$, $\bm{s}_i\neq \bm{s}_j$ and $i\neq j$. 
\end{remark}
%
%
%
\begin{corollary}\label{chifixedmanifold} 
Let $\{\bm{X}(\bm{s})$, $\bm{s}\in \mathcal{L}\}$ be a Gaussian $p$-dimensional ($p\geq2$) vector field with zero mean on a compact $m$-dimensional submanifold $\mathcal{L}\subset\mathbb{R}^n$ of positive reach and positive $m$-dimensional Lebesgue measure, such that $\{Y(\bm{t})$, $\bm{t}\in \mathcal{L} \times \mathbb{S}^{p-1}\}$ in (\ref{Yexpress}) satisfies assumption (A3). If $r_Y(\bm{t}_1,\bm{t}_2)<1$ for all $\bm{t}_1,\bm{t}_2$ from $\mathcal{L}\times \mathbb{S}^{p-1}$, $\bm{t}_1\neq\bm{t}_2$, then 
\begin{align}
\mathbb{P} \left(\sup_{\bm{s}\in \mathcal{L} } \|\bm{X}(\bm{s})\|> u\right) = \frac{H_{m,\alpha}}{(2\pi)^{(p-1)/2}} \int_{\mathcal{L}\times \mathbb{S}^{p-1}}  \|B_{\bm{t}} P_{\bm{s}}\|_{m} d\mathcal{H}_{m+p-1}(\bm{t})   u^{2m/\alpha + p-1} \Psi(u) (1+o(1)),
\end{align}
as $u\rightarrow\infty,$ where $P_{\bm{s}}$ is an $n\times m$ dimensional matrix whose columns are orthonormal and span the tangent space of $T_{\bm{s}}\mathcal{L}$. 
\end{corollary}

\begin{remark}\label{p1case}$\;$\\[-20pt]
\begin{itemize}
\item[a.] This corollary is a direct consequence of Theorem~\ref{fixedmanifold} using $R=(m,p-1)$. To see this, notice that $H_{R,\pmb{\alpha}}=H_{m,\alpha} H_{p-1,2}=H_{m,\alpha}(\sqrt{\pi})^{-(p-1)}$, because of the factorization lemma (see Remark~\ref{factorization}) and the well known fact $H_2=(\pi)^{-1/2}$ (see page 31, Piterbarg~\cite{Piterbarg:1996}). Also notice that $\|\frac{1}{\sqrt{2}}{\bm{I}}_{p} P_{{\bm{u}}}\|_{p-1}=2^{-(p-1)/2}$, where $P_{{\bm{u}}}$ is a $p\times (p-1)$ dimensional matrix whose columns span the tangent space of $T_{{\bm{u}}}\mathbb{S}^{p-1}$.
\item[b.] Even though the result in this corollary is stated for $p\geq2$, it can be easily extended to the case $p=1$. When $p=1$, we write $\bm{X}(\bm{s})=X(\bm{s})\in\mathbb{R}$ and $\mathbb{S}^{p-1}=\{\pm 1\}$.  Then using the same proof of this corollary, one can show that under the assumptions given in this corollary (in a broader sense such that $B_{\bm{t}}=B_{\bm{s}}$ only depends on $\bm{s}\in\mathcal{L}$, because $\mathbb{S}^{p-1}$ now is a discrete set), we have that as $u\rightarrow\infty$,
\begin{align}
\mathbb{P} \left(\sup_{\bm{s}\in \mathcal{L} } |X(\bm{s})|> u\right) = 2 H_{m,\alpha}  \int_{\mathcal{L}}  \|B_{\bm{s}} P_{\bm{s}}\|_{m} d\mathcal{H}_m(\bm{s})   u^{2m/\alpha} \Psi(u) (1+o(1)),
\end{align}
where the factor $2$ on the right-hand side is the cardinality of the set $\mathbb{S}^0$.
\end{itemize}
\end{remark}

{\bf Examples}. Below we give two examples of Gaussian vector fields $\bm{X}$ that satisfy assumption (A3). \\
(i) Let $X_{1}(\bm{s}),\cdots, X_{p}(\bm{s})$ be i.i.d. copies of $\{X(\bm{s}),\;\bm{s}\in \mathcal{L}\}$, which is assumed to be locally-$(n,\alpha,B_{\bm{s}})$-stationary, where $0<\alpha\leq 2$, that is,
\begin{align*}
r_X(\bm{s}_1,\bm{s}_2)= 1 - \|B_{\bm{s}}(\bm{s}_1-\bm{s}_2)\|^\alpha (1+o(1)),  
\end{align*}
as $\max\{\|\bm{s}-\bm{s}_1\|,\|\bm{s}-\bm{s}_2\|\}\rightarrow0$. In this case, (A3) is satisfied because
\begin{align*}
r_Y(\bm{t}_1,\bm{t}_2) = &r_X(\bm{s}_1,\bm{s}_2){\bm{v}}_1^T{\bm{v}}_2 \\
= &1- [\| B_{\bm{s}} (\bm{s}_1-\bm{s}_2)\|^\alpha + \frac{1}{2}\|{\bm{v}}_1-{\bm{v}}_2\|^2] (1+o(1)),
\end{align*}
$\max\{\|\bm{t}-\bm{t}_1\|,\|\bm{t}-\bm{t}_2\|\}\rightarrow0$. In other words, $Y(\bm{t})$ is locally-$(E,\pmb{\alpha},D_{\bm{t}})$-stationary, where $D_{\bm{t}} = \text{diag}(B_{\bm{s}} ,\frac{1}{2}{\bm{I}}_{p})$, $E=\{n,p\}$ and $\pmb{\alpha} = \{\alpha,2\}.$\\

(ii) Consider $X_{i}(\bm{s})$ as a locally-$(n,2,(A^{i,i}_{\bm{s}})^{1/2})$ stationary field, where $A^{i,i}_{\bm{s}}$ are positive definite $n\times n$ matrices, for $i=1,\cdots,p$. Also for $1\leq i\neq j\leq p$, suppose Cov$(X_i(\bm{s}_1),X_j(\bm{s}_2))=(\bm{s}_1-\bm{s}_2)^TA_{\bm{s}}^{i,j}(\bm{s}_1-\bm{s}_2)(1+o(1))$, as $\max\{\|\bm{s}-\bm{s}_1\|,\|\bm{s}-\bm{s}_2\|\}\rightarrow0$, where $A_{\bm{s}}^{i,j}$ are $n\times n$ symmetric matrices. So overall we may write
\begin{align*}
\text{Cov}(X_i(\bm{s}_1),X_j(\bm{s}_2)) = \delta_{ij} - (\bm{s}_1-\bm{s}_2)^TA_{\bm{s}}^{i,j}(\bm{s}_1-\bm{s}_2)(1+o(1)),
\end{align*}
as $\max\{\|\bm{s}-\bm{s}_1\|,\|\bm{s}-\bm{s}_2\|\}\rightarrow0$. Using (\ref{covstructure}), we have 
\begin{align*}
r_Y(\bm{t}_1,\bm{t}_2) = 1 - \frac{1}{2}\|{\bm{v}}_1-{\bm{v}}_2\|^2 - (\bm{s}_1-\bm{s}_2)^T \left\{\sum_{i=1}^p \sum_{j=1}^p [v_{i}v_{j} A_{\bm{s}}^{i,j} ] \right\}(\bm{s}_1-\bm{s}_2)(1+o(1)).
\end{align*}

Let $A_{\bm{t}}=\sum_{i=1}^p \sum_{j=1}^p [v_{i}v_{j} A_{\bm{s}}^{i,j} ]$. If $A_{\bm{t}}$ is positive definite, then (A3) is satisfied with $B_{\bm{t}}=(A_{\bm{t}})^{1/2}$, $E=n+p$ and $\pmb{\alpha}=2$. The matrix $A_{\bm{t}}$ is positive definite under many possible conditions. For example, if for each $i$, $\lambda_{\min}(A_{\bm{t}}^{i,i}) > \sum_{j\neq i} |\lambda_{\min}(A_{\bm{t}}^{i,j})|$, where $\lambda_{\min}$ is the smallest eigenvalue of a matrix, then $A_{\bm{t}}$ is positive definite because for any $u\in\mathbb{R}^{n}$ with $\|u\|>0$ and any ${\bm{v}}\in\mathbb{S}_{r-1}$,
\begin{align*}
u^T A_{\bm{t}}u & \geq \sum_{i=1}^p \sum_{j=1}^p \lambda_{\min}(A_{\bm{t}}^{i,j}) v_{i}v_{j}\|u\|^2 = {\bm{v}}^T\Lambda_{\min}{\bm{v}} \|u\|^2 >0,
\end{align*}
where $\Lambda_{\min}$ is a matrix consisting of $\lambda_{\min}(A_{\bm{t}}^{i,j})$, which is positive definite. \\

\begin{section}{Extremes of rescaled Gaussian and $\chi$ fields on manifolds}\label{MainResult}

In this section, we consider a class of centered Gaussian fields $\left\{Z_h(\bm{t}), \bm{t} \in \mathcal {M}_h\right\}_{h\in(0,h_0]}$ for some $0<h_0<1$, where $\mathcal {M}_h=\mathcal {M}_{h,1} \times \mathcal {M}_{h,2}$ are $r$-dimensional compact submanifolds of ${\mathbb R}^n$. The goal is to develop the result in (\ref{resultintro}), where the index $\bm{t}$ is partially rescaled by multiplying $h^{-1}$. For simplicity of exposition, in the structure $(E,\pmb{\alpha})$, we take $k=2$ so that $\pmb{\alpha}=(\alpha_1,\alpha_2)$, $E=(n_1,n_2)$ and $R=(r_1,r_2)$, where $1\leq r_1\leq n_1$, $1\leq r_2\leq n_2$, $r=r_1+r_2$, and $n=n_1+n_2$. The results in this section can be generalized to use the same structure $(E,\pmb{\alpha})$ as in Section~\ref{unrescaledM}. \\ 

%
We first give the following assumptions before formulating the main result. For $\bm{t} = (\bm{t}_{(1)}^T, \bm{t}_{(2)}^T)^T \in \mathbb{R}^{n_1}\times \mathbb{R}^{n_2}=\mathbb{R}^n$, let $\xi_h:\mathbb{R}^n\mapsto \mathbb{R}^n$ be a function such that $\xi_h(\bm{t})=(h\bm{t}_{(1)}^T, \bm{t}_{(2)}^T)^T$ and $\xi_h^{-1}$ be its inverse. Denote $\overbar{\mathcal{M}}_h = \xi_h^{-1}(\mathcal{M}_h)=\{\bm{t}:  \xi_h(\bm{t})\in \mathcal{M}_h\}$. 
Let $\overbar Z_h(\bm{t}) = Z_h(\xi_h(\bm{t})), \bm{t}\in\overbar{\mathcal{M}}_h$. Let $\bar r_h(\bm{t}_1, \bm{t}_2)$ be the covariance between $\overbar Z_h(\bm{t}_1)$ and $\overbar Z_h(\bm{t}_2)$, for $\bm{t}_1,\bm{t}_2\in\overbar{\mathcal{M}}_h$.
\begin{itemize}
\item[(B1)] Assume $\mathcal{M}_h = \mathcal {M}_{h,1} \times \mathcal {M}_{h,2}$, where $\mathcal {M}_{h,i}$ is a $r_i$-dimensional compact submanifold of $\mathbb{R}^{n_i}$, with $\inf_{0<h\leq h_0}\Delta(\mathcal{M}_{h,i})>0$, $i=1,2$, and 
\begin{align*}
0<  \inf_{0<h\leq h_0}\mathcal{H}_{r_i}(\mathcal{M}_{h,i}) \leq \sup_{0<h\leq h_0}\mathcal{H}_{r_i}(\mathcal{M}_{h,i})<\infty,\; i=1,2.
\end{align*}
\item[(B2)] $\overbar Z_h(\bm{t})$ is locally-$(E, \pmb{\alpha}, D_{\xi_h(\bm{t}),h})$-stationary in the following uniform sense: for $\bm{t},\bm{t}_1,\bm{t}_2\in\overbar{\mathcal{M}}_h$, as $\max\{\|\bm{t}-\bm{t}_1\|,\|\bm{t}-\bm{t}_2\|\}\rightarrow0$,
\begin{align}
\bar r_h(\bm{t}_1, \bm{t}_2)= 1 - | D_{\xi_h(\bm{t}),h} (\bm{t}_1-\bm{t}_2)|_{E,\pmb{\alpha}} (1+o(1)),
\end{align}
where the $o(1)$-term is uniform in $\bm{t}\in\overbar{\mathcal{M}}_h$ and $0<h\leq h_0$, and $D_{\bm{s},h}=\text{diag}(D_{\bm{s},h}^{(1)},D_{\bm{s},h}^{(2)})$, $\bm{s}\in\mathcal{M}_h$ is a block diagonal matrix. Here for $i=1,2$, the dimension of $D_{\bm{s},h}^{(i)}$ is $e_i\times e_i$, and the matrix-valued function $D_{\bm{s},h}^{(i)}$ of $\bm{s}$ has continuous components on $\mathcal{M}_h$. Also
\begin{align}\label{eigenbounds}
0 < \inf_{0<h\leq h_0, \bm{s}\in\mathcal{M}_h}\lambda_{\min}([D_{\bm{s},h}^{(i)}]^T D_{\bm{s},h}^{(i)}) \leq \sup_{0<h\leq h_0, \bm{s}\in\mathcal{M}_h} \lambda_{\max}([D_{\bm{s},h}^{(i)}]^T D_{\bm{s},h}^{(i)}) <\infty,\; i=1,2.
\end{align}
%
%
\item[(B3)] Suppose that, for any $x>0$, there exists $\eta > 0$ such that $Q(x) < \eta<1,$ where
\begin{align}\label{QDelta}
Q(x)=\sup_{0<h\leq h_0}\{|\bar r_h(\bm{t},\bm{s})|: \bm{t}, \bm{s}\in\overbar{\mathcal{M}}_{h}, \|\bm{t}^{(1)}-\bm{s}^{(1)}\|>x\}.
\end{align}
%
%
\item[(B4)] There exist $x_0> 0$ and a function $v(\cdot)$ such that for $x>x_0,$ we have
\begin{align}\label{SupGauss2}
Q(x)\Big|(\log x)^{2(r_1/\alpha_1+r_2/\alpha_2)}\Big|\leq v(x),
\end{align}
where $v$ is monotonically decreasing, such that, for any $q > 0,$ $v(x^q)=O(v(x))=o(1)$ and $v(x) x^{q}\to \infty$ as $x\rightarrow\infty$. 
\end{itemize}

\begin{remark}
Assumptions (B1)-(B3) extends their counterparts used in Theorem~\ref{fixedmanifold} to some forms that are uniform for the classes of Gaussian fields and manifolds. Assumption (B4) is analogous to the classical Berman condition used for proving extreme value distributions \cite{Berman:1964}. An example of $v(x)$ in assumption (B4) is given by $v(x)=(\log x)^{-\beta}$, for some $\beta>0$.
\end{remark}

\begin{theorem}\label{ProbMain}
Suppose assumptions (B1)-(B4) hold. 
Let
\begin{align}\label{BetaExp}
\beta_h=&\Big(2r_1\log\frac{1}{h} \Big)^{\frac{1}{2}}  +\Big(2r_1\log\frac{1}{h}\Big)^{-\frac{1}{2}} \nonumber\\
&\hspace{1cm}\times\bigg[\Big(\frac{r_1}{\alpha_1}+\frac{r_2}{\alpha_2}-\frac{1}{2}\Big)\log{\log\frac{1}{h} } +\log\bigg\{\frac{(2r_1)^{\frac{r_1}{\alpha_1}+\frac{r_2}{\alpha_2}-\frac{1}{2}}}{\sqrt{2\pi}}H_{R,\pmb{\alpha}}I_h(\mathcal {M}_h)\bigg\}\bigg], 
\end{align}
where $I_h(\mathcal {M}_h) = \int_{\mathcal {M}_{h}} \|D_{\bm{t},h} P_{\bm{t}}\|_{r_1} d\mathcal{H}_r(\bm{t})$ with $P_{\bm{t}}$ an $n\times r$ matrix with orthonormal columns spanning $T_{\bm{t}}\mathcal {M}_{h}$. Then
\begin{align}\label{ConclRes}
\lim_{h\rightarrow0}\mathbb{P}\left\{\sqrt{2r_1\log\tfrac{1}{h} } \left(\sup_{\bm{t}\in\mathcal {M}_{h}}Z_h(\bm{t})-\beta_h \right)\leq z\right\}= e^{-e^{-z}}.
\end{align}
\end{theorem}

\begin{remark}$\;$\\[-20pt]
\begin{itemize}
\item[a.] If there exists $\gamma>0$ such that $I_h(\mathcal{M}_{h}) \rightarrow \gamma$ as $h\rightarrow 0$. Then obviously $\gamma$ can replace $I_h(\mathcal{M}_{h})$ in the theorem. Also if $\mathcal{M}_h\equiv\mathcal{M}$ and $D_{\bm{t},h}\equiv D_{\bm{t}}$ (i.e. they are independent of $h$), then $I_h(\mathcal{M}_{h}) = \int_{\mathcal {M}} \|D_{\bm{t}} M_{\bm{t}}\|_{r_1} d\mathcal{H}_r(\bm{t})$. 
\item[b.] In fact, it can be easily seen from the proof that the result in the theorem also holds for the case that $\mathcal {M}_{h,1}=\emptyset$ (that is, $r_2=0$) so that $\mathcal{M}_{h}\equiv\mathcal {M}_{h,1}$. 
\end{itemize}
\end{remark}

Next we consider the asymptotic extreme value distribution of rescaled $\chi$-fields on manifolds. For some $0<h_0<1$, let $ \{\bm{X}_h(\bm{s}), \bm{s} \in \mathcal{L}_h\}_{h\in(0,h_0]}$ be a class of centered $p$-dimensional Gaussian random vector fields, where $\bm{X}_h=(X_{h,1},\cdots,X_{h,p})^T$ and $\mathcal{L}_h$ are $m$-dimensional compact submanifolds of ${\mathbb R}^n$ ($1\leq m\leq n$). 
Let ${\bm{v}}=(v_1,\cdots,v_p)^T\in\mathbb{R}^p$ and $\bm{t}=(\bm{s}^T, {\bm{v}}^T)^T\in\mathbb{R}^{n+p}$. Let 
\begin{align}\label{zhchi}
Z_h(\bm{t})=Z_h(\bm{s},{\bm{v}})=X_{h,1}(\bm{s})v_1+\cdots+X_{h,p}(\bm{s})v_p, \; \bm{t} \in \mathcal{M}_h := \mathcal{L}_h\times \mathbb{S}^{p-1}
\end{align}
Using the property of Euclidean norm, we have
\begin{align}\label{relation2}
\sup_{s\in\mathcal{L}_h}\|\bm{X}_h(\bm{s})\| = \sup_{\bm{t}\in\mathcal{M}_h} Z_h(\bm{t}).
\end{align}
%

 \begin{corollary}\label{rescaledchifield}
Suppose $p\geq2$ and $\{Z_h(\bm{t}), \;\bm{t}\in\mathcal{L}_h\times\mathbb{S}^{p-1}\}_{h\in(0,h_0]}$ in (\ref{zhchi}) satisfies assumptions (B1)-(B4) with $E=\{n,p\}$, $R=\{m,p-1\}$, $\pmb{\alpha}=\{\alpha,2\}$, and $D_{\bm{t},h}=\text{diag}(B_{\bm{t},h},\frac{1}{\sqrt{2}}{\bm{I}}_p)$ where $B_{\bm{t},h}$ is a nonsingular $n\times n$ dimensional matrix. Let
\begin{align}\label{BetaExp}
\beta_h=&\Big(2m\log\frac{1}{h} \Big)^{\frac{1}{2}}  +\Big(2m\log\frac{1}{h}\Big)^{-\frac{1}{2}}\bigg[\Big(\frac{m}{\alpha}+\frac{p-2}{2}\Big)\log{\log\frac{1}{h} } +\log\bigg\{\frac{(2m)^{\frac{m}{\alpha}+\frac{p-2}{2}}}{(\sqrt{2\pi})^p}H_{m,\alpha} I_h(\mathcal {M}_h)\bigg\}\bigg], 
\end{align}
where $I_h(\mathcal {M}_h) = \int_{\mathcal {L}_{h}\times \mathbb{S}^{p-1}} \|B_{\bm{t},h} P_{\bm{s}}\|_{m} d\mathcal{H}_{m+p-1}(\bm{t})$ with $P_{\bm{s}}$ an $n\times m$ matrix with orthonormal columns spanning $T_{\bm{s}}\mathcal {L}_{h}$. Then
\begin{align}\label{ConclRes}
\lim_{h\rightarrow0}\mathbb{P}\left\{\Big(2m\log\tfrac{1}{h} \Big)^{\frac{1}{2}} \left(\sup_{\bm{s}\in\mathcal {L}_{h}}\|\bm{X}_h(\bm{s})\|-\beta_h \right)\leq z\right\}= e^{-e^{-z}}.
\end{align}
\end{corollary}

\begin{remark}
The result in this corollary immediately follows from Theorem~\ref{ProbMain}. See Remark~\ref{p1case} (a) for some relevant calculation. Also, similar to Remark~\ref{p1case} (b), the result in this corollary can be extended to the case $p=1$, for which (\ref{ConclRes}) holds with 
\begin{align*}
\beta_h=&\Big(2m\log\frac{1}{h} \Big)^{\frac{1}{2}}  +\Big(2m\log\frac{1}{h}\Big)^{-\frac{1}{2}} \bigg[\Big(\frac{m}{\alpha}-\frac{1}{2}\Big)\log{\log\frac{1}{h} } +\log\bigg\{\frac{(2m)^{\frac{m}{\alpha}-\frac{1}{2}}}{\sqrt{2\pi}} H_{m,\alpha} I_h(\mathcal {M}_h)\bigg\}\bigg],
\end{align*}
where $I_h(\mathcal {M}_h) = 2\int_{\mathcal {L}_{h}} \|B_{\bm{s},h} P_{\bm{s}}\|_{m} d\mathcal{H}_m(\bm{s})$. 
\end{remark}


\section{Proofs}\label{proof}

\subsection{Geometric construction for the proof of Theorem~\ref{fixedmanifold}}\label{geoconstuct}
The proof of Theorem~\ref{fixedmanifold} relies on some geometric construction on manifolds with positive reach, which we present first. Let $M$ be a $r$-dimensional submanifold of $\mathbb{R}^n$. Suppose it has positive reach, i.e., $\Delta(M)>0$. 
%
For $\varepsilon,\eta>0$, a set of points $Q$ on $M$ is called a $(\varepsilon,\eta)$-sample, if \\[-20pt]
\begin{itemize}
\item[(i)] $\varepsilon$-covering: for any $x\in M$, there exists $y\in Q$ such that $\|x-y\|\leq \varepsilon$;\\[-20pt]
\item[(ii)] $\eta$-packing: for any $x,y\in Q$, $\|x-y\| > \eta$.\\[-20pt]
\end{itemize}
For simplicity, we alway use $\eta=\varepsilon$, and such an $(\varepsilon,\varepsilon)$-sample is called an $\varepsilon$-net. It is known that an $\varepsilon$-net always exists for any positive real $\varepsilon$ when $M$ is bounded (Lemma 5.2, Boissonnat, Chazal and Yvinec~\cite{Boissonnat:2018}). Let $N_\varepsilon$ be the cardinality of this $\varepsilon$-net. Let 
\begin{align*}
&P_\varepsilon = \max\{n: \; \text{there exists an } \varepsilon\text{-packing of } M \text{ of size } n\},\\
&C_\varepsilon = \min\{n: \; \text{there exists an } \varepsilon\text{-covering over } M \text{ of size } n\},
\end{align*}
which are called the $\varepsilon$-packing and $\varepsilon$-covering numbers, respectively. It is known that (see Lemma 5.2 in Niyogi et al.~\cite{Niyogi:2008}) $$P_{2\varepsilon} \leq C_\varepsilon \leq N_\varepsilon \leq P_{\varepsilon}.$$

Also it is given on page 431 of Niyogi et al. (2008) that when $\varepsilon<\Delta(M)/2$
\begin{align*}
P_\varepsilon \leq \frac{\mathscr{H}_r(M)}{[\cos^r (\theta)]\varepsilon^r B_r},
\end{align*}
where $B_r$ is the volume of the unit $r$-ball, and $\theta=\arcsin(\varepsilon/2).$ This implies that $N_\varepsilon = O(\varepsilon^{-r})$, as $\varepsilon\rightarrow0$, when $\mathscr{H}_r(M)$ is bounded. \\

Let $\{x_1,\cdots,x_{N_\varepsilon}\}\subset M$ be an $\varepsilon$-net. With this $\varepsilon$-net, we can construct a Voronoi diagram restricted on $M$ consisting of $N_\varepsilon$ Voronoi cells $V_1,\cdots,V_{N_\varepsilon}$, where $V_i=\{x\in M: \|x - x_i \| \leq \|x - x_j\|,\; \text{for all }j\neq i\}$. The Voronoi diagram gives a partition of $M$, that is $M = \cup_{i=1}^{N_\varepsilon} V_i$. Due to the definition of the $\varepsilon$-net, we have that 
\begin{align*}
(\mathcal{B}(x_i,\varepsilon/2)\cap M) \subset V_i \subset (\mathcal{B}(x_i,\varepsilon)\cap M),\; i=1,\cdots N_\varepsilon.
\end{align*}
In other words, the shape of all the Voronoi cells is always not very thin.
\subsection{Proof of Theorem~\ref{fixedmanifold}}
We first give a lemma used in the proof of Theorem~\ref{fixedmanifold}.
\begin{lemma}\label{fixedanxillary}
Suppose that the conditions in Theorem~\ref{fixedmanifold} hold. For any subset $U\subset \mathcal{M}$, if there exists a diffeomorphism $\psi:U\mapsto \Omega\subset\mathbb{R}^r$, where $\Omega=\psi(U)$ is a closed Jordan set of positive $r$-dimensional Lebesgue measure, then as $u\rightarrow\infty$
\begin{align}\label{subsetextreme}
\mathbb{P} \left(\sup_{\bm{t}\in U} X(\bm{t}) > u\right) = H_{R,\pmb{\alpha}} \int_{U} \prod_{j=1}^k \|D_{j,\bm{t}} P_{j,\bm{t}}\|_{r_j} d\mathcal{H}_r(\bm{t})  \prod_{i=1}^k u^{2r_i/\alpha_i} \Psi(u) (1+o(1)).
\end{align}
\end{lemma}
\begin{proof}
Let $\wt X = X\circ \psi^{-1}$, which is a Gaussian field indexed by $\Omega\subset\mathbb{R}^r$. Consider $\wt{\bm{t}},\wt{\bm{t}}_1,\wt{\bm{t}}_2\in\Omega$ such that $\max\{\|\wt{\bm{t}}-\wt{\bm{t}}_1\|,\|\wt{\bm{t}}-\wt{\bm{t}}_2\|\}\rightarrow0$. Since $\psi$ is a differomphism, we also have $\max\{\|\psi^{-1}(\wt{\bm{t}})-\psi^{-1}(\wt{\bm{t}}_1)\|,\|\psi^{-1}(\wt{\bm{t}})-\psi^{-1}(\wt{\bm{t}}_2)\|\}\rightarrow0$. Let $J_{\psi^{-1}}$ be the Jacobian matrix of $\psi^{-1}$, whose dimension is $n\times r$. Using assumption (A1), we have
\begin{align*}
\text{Cov} (\wt X(\wt{\bm{t}}_1), \wt X(\wt{\bm{t}}_2)) = & \text{Cov}(X(\psi^{-1}(\wt{\bm{t}}_1)), X(\psi_i^{-1}(\wt{\bm{t}}_2))) \\
= & 1 - | D_{\psi^{-1}(\wt{\bm{t}})} (\psi^{-1}(\wt{\bm{t}}_1)-\psi^{-1}(\wt{\bm{t}}_2))|_{E,\pmb{\alpha}} (1+o(1)) \\
= & 1 - | D_{\psi^{-1}(\wt{\bm{t}})} J_{\psi^{-1}}(\wt{\bm{t}})(\wt{\bm{t}}_1- \wt{\bm{t}}_2)|_{E,\pmb{\alpha}} (1+o(1)),
\end{align*}
where in the last step we have used a Taylor expansion. Since the columns of the Jacobian matrix $J_{\psi^{-1}}$ span the tangent space $T_{\psi^{-1}(\wt{\bm{t}})} \mathcal{M}$, and the matrix $D_{\psi^{-1}(\wt{\bm{t}})}$ is assumed to be nonsingular, the matrix $D_{\psi^{-1}(\wt{\bm{t}})} J_{\psi^{-1}}(\wt{\bm{t}})$ is of full rank, and therefore 
\begin{align*}
A(\wt{\bm{t}}) : = [J_{\psi^{-1}}(\wt{\bm{t}})]^T[D_{\psi^{-1}(\wt{\bm{t}})}]^T D_{\psi^{-1}(\wt{\bm{t}})} J_{\psi^{-1}}(\wt{\bm{t}})
\end{align*}
is positive definite. Also note that $A(\wt{\bm{t}})$ is block diagonal matrix, where the diagonal blocks have dimension $r_i\times r_i$, $i=1,\cdots,k$. Let $A(\wt{\bm{t}})^{1/2}$ be the principal square root matrix of $A(\wt{\bm{t}})$. We have that 
\begin{align*}
\text{Cov} (\wt X(\wt{\bm{t}}_1), \wt X(\wt{\bm{t}}_2)) = 1 - |A(\wt{\bm{t}})^{1/2}(\wt{\bm{t}}_1- \wt{\bm{t}}_2)|_{R,\pmb{\alpha}} (1+o(1)).
\end{align*}
Using Theorem 7.1 in Piterbarg~\cite{Piterbarg:1996}, we obtain that $u\rightarrow\infty$
\begin{align*}
\mathbb{P} \left(\sup_{\wt{\bm{t}}\in \Omega} \wt X(\wt{\bm{t}})> u\right) = H_{R,\pmb{\alpha}} \int_{\Omega} \det [A(\wt{\bm{t}})^{1/2}] d\mathcal{H}_r(\wt{\bm{t}})  \prod_{i=1}^k u^{2r_i/\alpha_i} \Psi(u) (1+o(1)).
\end{align*}
Using the area formula on manifolds (see page 117, Evans and Gariepy~\cite{Evans:1992}) and noticing that $\sup_{\wt{\bm{t}}\in \Omega} \wt X(\wt{\bm{t}}) = \sup_{\bm{t}\in U} X(\bm{t})$, we have 
\begin{align*}
\mathbb{P} \left(\sup_{\bm{t}\in U} X(\bm{t}) > u\right) = H_{R,\pmb{\alpha}} \int_{U} \frac{\det [A(\psi(\bm{t}))^{1/2}]}{\det[B(\psi(\bm{t}))^{1/2}]} d\mathcal{H}_r(\bm{t})  \prod_{i=1}^k u^{2r_i/\alpha_i} \Psi(u) (1+o(1)),
\end{align*}
where $B(\psi(\bm{t}))=[J_{\psi^{-1}}(\psi(\bm{t}))]^TJ_{\psi^{-1}}(\psi(\bm{t}))$. Let $\{p_1(\bm{t}), \cdots,p_r(\bm{t})\}$ be an orthonormal basis of the tangent space $T_{\bm{t}} \mathcal{M}$ and write $P_{\bm{t}}=[p_1(\bm{t}), \cdots,p_r(\bm{t})]$. There exists a $r\times r$ nonsingular matrix $Q_{\bm{t}}$ such that $J_{\psi^{-1}}(\psi(\bm{t})) = P_{\bm{t}}Q_{\bm{t}}$. Hence 
\begin{align*}
\frac{\det [A(\psi(\bm{t}))^{1/2}]}{\det[B(\psi(\bm{t}))^{1/2}]} = \frac{\det [Q_{\bm{t}}] \det[(P_{\bm{t}}^T D_{\bm{t}}^T D_{\bm{t}} P_{\bm{t}})^{1/2}]}{\det [Q_{\bm{t}}] } = \det[(P_{\bm{t}}^T D_{\bm{t}}^T D_{\bm{t}} P_{\bm{t}})^{1/2}]
\end{align*}

For $j=1,\cdots,k$, let $P_{j,\bm{t}}$ be a $e_j\times r_j$ matrix whose columns span the tangent space of $M_j$. Then by the Cauchy-Binet formula (see Broida and Williamson~\cite{Broida:1989}, page 214), we have
\begin{align}\label{determine}
\det[P_{\bm{t}}^T D_{\bm{t}}^T D_{\bm{t}} P_{\bm{t}}]^{1/2} = \prod_{j=1}^k \det[(P_{j,\bm{t}}^T D_{j,\bm{t}}^T D_{j,\bm{t}} P_{j,\bm{t}})^{1/2}] = \prod_{j=1}^k \|D_{j,\bm{t}} P_{j,\bm{t}}\|_{r_j}.
\end{align}
Therefore we get (\ref{subsetextreme}).
\end{proof}

{\bf Proof of Theorem~\ref{fixedmanifold}}
\begin{proof}
For any $\bm{t}\in\mathcal{M}$, let $\rho\equiv\rho_{\bm{t}}: \mathcal{B}(\bm{t},\epsilon)\cap \mathcal{M}\mapsto T_{\bm{t}}\mathcal{M}$ be the projection map to the tangent space $T_{\bm{t}}\mathcal{M}$, that is, $\rho$ is a restriction of the normal projection $\pi_{T_{\bm{t}}\mathcal{M}}$ to the set $\mathcal{B}(\bm{t},\epsilon)\cap \mathcal{M}$. When $\epsilon<\Delta(\mathcal{M})/2$, it is known that $\rho$ is a diffeomorphism (see Lemma 5.4, Niyogi et al.~\cite{Niyogi:2008}). The Jacobian of $\rho$, denoted by $J_\rho$, is a differential map that projects the tangent space of $\mathcal{B}(\bm{t},\epsilon)\cap \mathcal{M}$ at any point in it onto $T_{\bm{t}}\mathcal{M}$. It is also known that the angles between two tangent spaces $T_{\bm{p}}\mathcal{M}$ and $T_{\bm{q}}\mathcal{M}$ is bounded by $L\|{\bm{p}}-{\bm{q}}\|$ for ${\bm{p}},{\bm{q}}\in \mathcal{B}(\bm{t},\epsilon) \cap \mathcal{M}$ when $\epsilon<\Delta(\mathcal{M})/2$ (see Propositions 6.2 and 6.3 of Niyogi et al.~\cite{Niyogi:2008}), where $L>0$ is a constant only depending on $\Delta(\mathcal{M})$. Hence $J_\rho$ is Liptschtiz continuous on $\mathcal{B}(\bm{t},\epsilon)\cap \mathcal{M}$. Suppose that $\{{\bm{e}}_1,\cdots,{\bm{e}}_r\}$ is an orthonormal basis of $T_{\bm{t}}\mathcal{M}$. Let $\iota: T_{\bm{t}}\mathcal{M}\mapsto \mathbb{R}^r$ be a map such that $\iota({\bm{y}})=(y_1,\cdots,y_r)\in\mathbb{R}^r$ for ${\bm{y}}=y_1{\bm{e}}_1+\cdots y_r{\bm{e}}_r\in T_{\bm{t}}\mathcal{M}$. Then $\psi:=\iota\circ\rho$ is the diffeomorphism we need to apply Lemma~\ref{fixedanxillary}. \\

We choose $\epsilon<\Delta(\mathcal{M})/10$. Using the method in Section~\ref{geoconstuct}, we find an $\epsilon$-net $\{\bm{t}_1,\cdots,\bm{t}_{N_\epsilon}\}$ for $\mathcal{M}$, and construct a partition of $\mathcal{M}$ with Voronoi cells $V_1,\cdots,V_{N_\epsilon}$, where $N_\epsilon = O(\epsilon^{-r})$. Since $V_i \subset (\mathcal{B}(\bm{t}_i,\epsilon)\cap \mathcal{M})$, $\rho\equiv\rho_{\bm{t}_i}$ is a diffeomorphism on $V_i$, $i=1,\cdots,N_\epsilon$.\\

Using Lemma~\ref{fixedanxillary}, we have that 
\begin{align*}
\mathbb{P} \left(\sup_{\bm{t}\in V_i} X(\bm{t}) > u\right) = H_{R,\pmb{\alpha}} \int_{V_i} \prod_{j=1}^k \|D_{j,\bm{t}} P_{j,\bm{t}}\|_{r_j} d\mathcal{H}_r(\bm{t})  \prod_{j=1}^k u^{2r_j/\alpha_j} \Psi(u) (1+o(1)),
\end{align*}
as $u\rightarrow\infty$, and hence
\begin{align}\label{TotalSum}
\sum_{i=1}^{N_\epsilon} \mathbb{P} \left(\sup_{\bm{t}\in V_i} X(\bm{t}) > u\right) = H_{R,\pmb{\alpha}} \int_{\mathcal{M}} \prod_{j=1}^k \|D_{j,\bm{t}} P_{j,\bm{t}}\|_{r_j} d\mathcal{H}_r(\bm{t})  \prod_{j=1}^k u^{2r_j/\alpha_j} \Psi(u) (1+o(1)).
\end{align}

Using the Bonferroni inequality, we have
\begin{align}\label{DoubleSum}
& \sum_{i=1}^{N_\epsilon} \mathbb{P} \left(\sup_{\bm{t}\in V_i} X(\bm{t}) > u\right)   -  \sum_{i\neq j} \mathbb{P} \left(\sup_{\bm{t}\in V_i} X(\bm{t}) > u,\; \sup_{\bm{t}\in V_j} X(\bm{t}) > u \right) \nonumber \\
 & \hspace{5cm} \leq  \mathbb{P} \left(\sup_{\bm{t}\in \mathcal{M}} X(\bm{t}) > u\right) \leq  \sum_{i=1}^{N_\epsilon} \mathbb{P} \left(\sup_{\bm{t}\in V_i} X(\bm{t}) > u\right)  .
\end{align}

For $i\neq j$, define $d_{\max}(V_i,V_j)=\sup\{\|x-y\|:\; x\in V_i, \;y\in V_j\}$ and $d_{\min}(V_i,V_j)=\inf\{\|x-y\|:\; x\in P_i, \;y\in P_j\}$. We divide the set of indices $S=\{(i,j):\;1\leq i\neq j\leq N_\epsilon\}$ into $S_1$ and $S_2$, where $S_1=\{(i,j)\in S:\; d_{\max}(V_i,V_j)\leq 5\epsilon\}$ and $S_2=\{(i,j)\in S:\; d_{\max}(V_i,V_j)\}> 5\epsilon\}$. If $(i,j)\in S_1$, then there exists $\bar{\bm{t}}\in\mathcal{M}$ such that $(V_i\cup V_j)\subset (\mathcal{B}(\bar{\bm{t}},5\epsilon)\cap\mathcal{M})\subset (\mathcal{B}(\bar{\bm{t}},\Delta(\mathcal{M})/2)\cap\mathcal{M}) $, and therefore using Lemma~\ref{fixedanxillary}, we have as $u\rightarrow\infty$
\begin{align*}
& \mathbb{P} \left(\sup_{\bm{t}\in V_i} X(\bm{t}) > u,\; \sup_{\bm{t}\in V_j} X(\bm{t}) > u \right) \\
= & \mathbb{P} \left(\sup_{\bm{t}\in V_i} X(\bm{t}) > u \right) + \mathbb{P} \left(\sup_{\bm{t}\in V_j} X(\bm{t}) > u \right) - \mathbb{P} \left(\sup_{\bm{t}\in V_i\cup V_j} X(\bm{t}) > u \right) \\
=& o(1) H_{R,\pmb{\alpha}} \int_{V_i\cup V_j} \prod_{j=1}^k \|D_{j,\bm{t}} P_{j,\bm{t}}\|_{r_j} d\mathcal{H}_r(\bm{t})  \prod_{j=1}^k u^{2r_j/\alpha_j} \Psi(u).
\end{align*}

Therefore as $u\rightarrow\infty$
\begin{align}\label{SubSumBound2}
\sum_{(i,j)\in S_1}\mathbb{P} \left(\sup_{\bm{t}\in V_i} X(\bm{t}) > u,\; \sup_{\bm{t}\in V_j} X(\bm{t}) > u \right) = o\left( \prod_{i=1}^k u^{2r_i/\alpha_i} \Psi(u)\right).
\end{align}

Next we proceed to consider $(i,j)\in S_2$. Let $Y(\bm{t},\bm{s})=X(\bm{t})+X(\bm{s})$. Note that 
\begin{align}\label{IntersectionBound}
\mathbb{P}\left(\sup_{\bm{t}\in V_i} X(\bm{t}) > u,\; \sup_{\bm{t}\in V_j} X(\bm{t}) > u \right)&\leq 
\mathbb{P}\left(\sup_{\bm{t}\in V_i,\bm{s}\in V_j} Y(\bm{t},\bm{s})>2u\right).
\end{align}
In order to further bound the probability on the right-hand side, we will use the Borell inequality \cite{Borell:1975} (see Theorem D.1 in Piterbarg~\cite{Piterbarg:1996}).
%
%
%
%
%
%
%
%
Notice that $d_{\min}(V_i,V_j) \geq d_{\max}(V_i,V_j) - 4\epsilon$, 
and hence
\begin{align*}
\min_{(i,j)\in S_2}d_{\min}(V_i,V_j) \geq \epsilon.
\end{align*}

%
The assumption in the theorem guarantees that $\rho:=\sup_{\|\bm{t}-\bm{s}\|\geq\epsilon}r_X(\bm{t},\bm{s})<1.$
This then yields that 
\begin{align*}
\max_{(i,j)\in S_2}\sup_{(\bm{t},\bm{s})\in V_i\times V_j} \Var\left(Y(\bm{t},\bm{s})\right)\leq 2+2\rho
\end{align*}
and\\[-20pt]
\begin{align*}
\sup_{(i,j)\in S_2}\sup_{(\bm{t},\bm{s})\in V_i\times V_j}\mathbb{E}\left(Y(\bm{t},\bm{s})\right)=0.
\end{align*}
Now it remains to show that $\mathbb{P}\left(\sup_{\bm{t}\in V_i,\bm{s}\in V_j} Y(\bm{t},\bm{s})>b\right) \leq 1/2$ for some constant $b$ for all $(i,j)\in S_2$ in order to apply the Borell inequality to $Y(\bm{t},\bm{s})$. Such $b$ exists because 
\begin{align*}
&\mathbb{P}\left(\sup_{\bm{t}\in V_i,\bm{s}\in V_j} Y(\bm{t},\bm{s})>u\right) \leq \mathbb{P}\left(\sup_{\bm{t}\in \mathcal{M},\bm{s}\in \mathcal{M} } Y(\bm{t},\bm{s})>u\right) \leq \mathbb{P}\left(\sup_{\bm{t}\in \mathcal{M}} X(\bm{t})>u/2\right)\\
\leq &H_{R,\pmb{\alpha}} \int_{\mathcal{M}} \prod_{j=1}^k \|D_{j,\bm{t}} P_{j,\bm{t}}\|_{r_j} d\mathcal{H}_r(\bm{t})  \prod_{j=1}^k \left(\frac{u}{2}\right)^{2  r_j/\alpha_j} \Psi\left(\frac{u}{2}\right) (1+o(1)),
\end{align*}
which tends to zero as $u\rightarrow\infty$. %
%
%
%
The application of the Borell inequality now gives that 
\begin{align}\label{PSupBound}
\mathbb{P}\left(\sup_{\bm{t}\in V_i,\bm{s}\in V_j} Y(\bm{t},\bm{s})>2u\right)&\leq 2\bar\Phi\bigg(\frac{u-b/2}{\sqrt{(1+\rho)/2}}\bigg). 
\end{align}
Also note that the cardinality $|S_2| \leq N_\epsilon^2 \leq C\epsilon^{-2r}$, 
for some constant $C>0$. Hence
%
\begin{align}\label{SubSumBound1}
\sum_{(i,j)\in S_2 } \mathbb{P}\left(\sup_{\bm{t}\in V_i} X(\bm{t}) > u,\; \sup_{\bm{t}\in V_j} X(\bm{t}) > u \right)
\leq &  2|S_2| \bar\Phi\bigg(\frac{u-b/2}{\sqrt{(1+\rho)/2}}\bigg)
=o\left( \prod_{i=1}^k u^{2r_i/\alpha_i} \Psi(u)\right),
\end{align}
%
as $u\rightarrow\infty$. Combining (\ref{TotalSum}), (\ref{DoubleSum}), (\ref{SubSumBound2}) and (\ref{SubSumBound1}), we have the desired result.
%
%
%
\end{proof}

\subsection{Geometric construction for the proof of Theorem~\ref{ProbMain}}
We first give some geometric construction used in the proof of Theorem~\ref{ProbMain}.\\

(i) {\em Voronoi diagram on ${\mathcal{M}_h}$}: Let $\ell_1=\inf_{h\in(0,h_0]}\Delta(\mathcal {M}_{h,1})/2$. It is known from Section~\ref{geoconstuct} that there exists an $(h\ell_1)$-net $\{\bm{s}_1,\cdots,\bm{s}_{m_h}\}$ on $\mathcal{M}_{h}$, where $m_h=O((h\ell_1)^{-{r_1}})$ is the cardinality of the net. With this $(h\ell_1)$-net and using the technique described in Section~\ref{geoconstuct}, we construct a Voronoi diagram restricted on $\mathcal {M}_{h,1}$. The collections of the cells are denoted by $\{J_{k,h}: k=1,\cdots,m_h\}$, which forms a partition of $\mathcal {M}_{h,1}$. Similarly for $\mathcal {M}_{h,2}$, with $\ell_2=\inf_{h\in(0,h_0]}\Delta(\mathcal {M}_{h,2})/2$, there exists an $\ell_2$-net $\{\bm{u}_1,\cdots,\bm{u}_{n_h}\}$ on $\mathcal {M}_{h,2}$, where $n_h=O(\ell_2^{-r_2})$. The cells of the corresponding Voronoi diagram on $\mathcal {M}_{h,2}$ are denoted by $U_1,\cdots,U_{n_h}$.\\


(ii) {\em Separation of Voronoi cells:} The construction of the Voronoi diagram restricted on ${\mathcal{M}_h}^{(1)}$ guarantees that each cell $J_{k,h} \supset  (\mathcal {M}_{h,1}\cap \mathcal{B}(\bm{s}_k,(h\ell_1)/2))$. In other words, $J_{k,h}$ is not too thin. For $0<\delta<\ell_1/2$, let $\partial \mathcal{J}_h = \cup_{k=1}^{m_h} (\partial J_{k,h})$ be the union of all the boundaries of the cells. Let 
\begin{align*}
\mathcal {B}^{h\delta} = \{x\in\mathcal{M}_{h}: d(x,\partial \mathcal{J}_h) \leq h\delta\},
\end{align*} 
which is the $(h\delta)$-enlarged neighborhood of $\partial \mathcal{J}_h$. We obtain $J_{k,h}^\delta = J_{k,h}\backslash \mathcal {B}^{h\delta}$ and $J_{k,h}^{-\delta} = J_{k,h}\backslash J_{k,h}^\delta$ for $1\leq k\leq m_h$. The geometric construction ensures that if $k\neq k^\prime$, $J_{k,h}^\delta$ and $J_{k^\prime,h}^\delta$ are separated by $\mathcal {B}^{h\delta}$, which is partitioned as $\{J_{k,h}^{-\delta},\; k=1,\cdots,m_h\}$ .\\
%

(iii) {\em Discretization:} We construct a dense grid on $\mathcal{M}_h$ as follows. Let $\Pi_{k,j}=(\Pi_{\bm{s}_k},\Pi_{\bm{u}_j})$ be the projection map from $J_{k,h}\times U_j$ to the tangent space $T_{\bm{s}_k}\mathcal {M}_{h,1} \times T_{\bm{u}_j} \mathcal {M}_{h,2}$. Let the image of $J_{k,h}\times U_j$ be $\wt J_{k,h}\times \wt U_j$. The choice of the $\ell_1$ and $\ell_2$ guarantees that $\Pi_{k,j}$ is a homeomorphism. Let $\{M_{\bm{s}_k}^i: i=1,\cdots,r_1\}$ be orthonormal vectors spanning the tangent space $T_{\bm{s}_k}\mathcal {M}_{h,1}$. For a given $\gamma > 0,$ consider the (discrete) set $\wt\Xi_{h\gamma \theta^{-2/\alpha_1}}(\wt J_{k,h})=\{\bm{t}\in \wt J_{k,h}: \bm{t}=\bm{s}_k+(h\gamma \theta^{-2/\alpha_1})\sum_{i=1}^{r_1}(e_iM_{\bm{s}_k}^i), e_i\in\mathbb{Z}\}$ and let $\Xi_{h\gamma \theta^{-2/\alpha_1}}(J_{k,h})=\Pi_{\bm{s}_k}^{-1}(\wt\Xi_{h\gamma \theta^{-2/\alpha_1}}(\wt J_{k,h}))$, which is a subset of $J_{k,h}$. Similarly, let $\{M_{\bm{u}_j}^i: i=1,\cdots,r_2\}$ be orthonormal vectors spanning the tangent space $T_{\bm{u}_j}\mathcal {M}_{h,2}$ and we discretize $\wt U_j$ with $\wt\Xi_{\gamma \theta^{-2/\alpha_2}}(\wt U_j)=\{\bm{v}\in\wt U_j: \bm{v}=\bm{u}_j+\sum_{i=1}^{r_2}e_i\gamma \theta^{-2/\alpha_2}M_{\bm{u}_j}^i, e_i\in\mathbb{Z}\}$ and denote $\Xi_{\gamma \theta^{-2/\alpha_2}}(U_j) = \Pi_{\bm{u}_j}^{-1} (\wt\Xi_{\gamma \theta^{-2/\alpha_2}}(\wt U_j))$. \\ 

We denote the union of all the grid points by
\begin{align}\label{fullgrid}
\Gamma_{h,\gamma,\theta} & = \cup_{k=1}^{m_h}\cup_{j=1}^{n_h} [\Xi_{h\gamma \theta^{-2/\alpha_1}}(J_{k,h}) \times \Xi_{\gamma \theta^{-2/\alpha_2}}(U_j)] \\
& = [\cup_{k=1}^{m_h} \Xi_{h\gamma \theta^{-2/\alpha_1}}(J_{k,h})]\times [\cup_{j=1}^{n_h} \Xi_{\gamma \theta^{-2/\alpha_2}}(U_j)].
%
\end{align} 

Let $N_{h}^{(1)}$ be the cardinality of the set $\cup_{k=1}^{m_h} \Xi_{h\gamma \theta^{-2/\alpha_1}}(J_{k,h})$. Then obviously,

\begin{align*}
N_{h}^{(1)} = |\cup_{k=1}^{m_h} \wt\Xi_{h\gamma \theta^{-2/\alpha_1}}(\wt J_{k,h})| = O\left(  \frac{\sum_{k=1}^{m_h} \mathcal{H}_{r_1}(\wt J_{k,h})}{(h\gamma\theta^{-2/\alpha_1})^{r_1}}\right) & = O\left(  \frac{\mathcal{H}_{r_1}(\mathcal {M}_{h,1})}{(h\gamma\theta^{-2/\alpha_1})^{r_1}}\right) \\
&=O(\theta^{2r_1/\alpha_1} h^{-r_1}\gamma^{-r_1}).
\end{align*}
Similarly, the cardinality of $\cup_{j=1}^{n_h} \Xi_{\gamma \theta^{-2/\alpha_2}}(U_j)$ is given by
\begin{align}\label{nustar}
N_{h}^{(2)}: = |\cup_{j=1}^{n_h} \Xi_{\gamma \theta^{-2/\alpha_2}}(U_j)| = O(\theta^{2r_2/\alpha_2} \gamma^{-r_2}).
\end{align}

It is easy to see that $(\mathcal{J}_{h}^\delta \times \mathcal {M}_{h,2}) \cap \Gamma_{h,\gamma,\theta} = [\cup_{k=1}^{m_h} \Xi_{h\gamma \theta^{-2/\alpha_1}}(J_{k,h}^\delta)]\times [\cup_{j=1}^{n_h} \Xi_{\gamma \theta^{-2/\alpha_2}}(U_j)],$ 
and
\begin{align}\label{nhdeltastar}
N_{h,\delta}^{(1)}: =|\cup_{k=1}^{m_h} \Xi_{h\gamma \theta^{-2/\alpha_1}}(J_{k,h})| = O(N_{h}^{(1)}) = O(\theta^{2r_1/\alpha_1} h^{-r_1}\gamma^{-r_1}).
\end{align}

\subsection{Proof of Theorem~\ref{ProbMain}}\label{proofmain}
For a random process or field $X(t)$, $t\in \mathcal{S}\subset\mathbb{R}^n$ and $\theta\in\mathbb{R}$, we denote
\begin{align*}
& \mathbb{P}_X(\theta, \mathcal{S}) = \mathbb{P} (\sup_{t\in \mathcal{S}} X(t) \leq \theta),\\
& \mathbb{Q}_X(\theta, \mathcal{S}) = 1- \mathbb{P}_X(\theta, \mathcal{S}) .
\end{align*}
With $\beta_h$ in (\ref{BetaExp}), let
\begin{align}\label{ThetaExp}
\theta_{h,z} &=\beta_h+\frac{1}{\sqrt{2r_1\log(1/h)}}z.
\end{align}
With this notation, we can rewrite (\ref{ConclRes}) as
\begin{align*}
\lim_{h\rightarrow0}\mathbb{P}_{Z_h}(\theta_{h,z}, \mathcal{M}_h)= e^{-e^{-z}}.
\end{align*}

To prove Theorem~\ref{ProbMain}, we need to establish a sequence of approximations using the above geometric construction, detailed in Lemmas~\ref{lemma3.1}-\ref{lemma3.6} as follows.\\


Recall that $I_h(\mathcal{A}) = \int_{\mathcal{A}}\|D_{\bm{t},h} P_{\bm{t}}\|_{r_1} d\mathcal{H}_r(\bm{t})$ for any measurable set $\mathcal{A}\subset \mathcal{M}_h$. In the following lemma we consider $\theta$ as a large number with $\theta = \theta_{h,z}$ as a special case in mind.

\begin{lemma}\label{lemma3.1}
For any $\epsilon>0$, there exist $\theta_0>0$ such that for all $\theta\geq\theta_0$, $0<h\leq h_0$, and $J_k\in \{J_{k,h}, J_{k,h}^\delta, J_{k,h}^{-\delta}\}$ with $1\leq k\leq m_h(J)$, we have for some $\epsilon_{k,h}$ with $|\epsilon_{k,h}|\leq \epsilon$,
\begin{align}\label{SumToInt}
\frac{\mathbb{Q}_{Z_h}(\theta, J_k\times \mathcal {M}_{h,2} )}{\theta^{2(r_1/\alpha_1+r_2/\alpha_2)}\Psi(\theta)}
=(1+ \epsilon_{k,h})h^{-r_1}\;
H_{R,\pmb{\alpha}} I_h(J_k \times\mathcal {M}_{h,2}).
\end{align}
%
\end{lemma}

\begin{proof}
For $J_k\in \{J_{k,h}, J_{k,h}^\delta, J_{k,h}^{-\delta}\}$, denote $\overbar J_k = \{\bm{t}_{(1)}/h: \bm{t}_{(1)}\in J_k\}$. Then notice that $\overbar J_k$ has a positive diameter and volume. Recall that $\xi_h(\bm{t})=(h\bm{t}_{(1)}^T, \bm{t}_{(2)}^T)^T$ for $\bm{t} = (\bm{t}_{(1)}^T, \bm{t}_{(2)}^T)^T \in \overbar J_k\times \mathcal {M}_{h,2}$ and the Gaussian field $\overbar Z_h(\bm{t}) = Z_h(\xi_h(\bm{t}))$ is locally-$(E,\pmb{\alpha},D_{\xi_h(\bm{t}),h})$-stationary on $\overbar J_k\times \mathcal {M}_{h,2}$. Let $\overbar I_h(\mathcal{A}) = \int_{\mathcal{A}}\|D_{\xi_h(\bm{t}),h} P_{\bm{t}}\|_{r_1} d\mathcal{H}_r(\bm{t})$ for any measurable set $\mathcal{A}\subset \xi_h^{-1}(\mathcal{M}_h)$. Then using Theorem~\ref{fixedmanifold}, we obtain that 
\begin{align*}
\frac{\mathbb{Q}_{\overbar Z_h}(\theta, \overbar J_k\times \mathcal {M}_{h,2} )}{\theta^{2(r_1/\alpha_1+r_2/\alpha_2)}\Psi(\theta)}
= H_{R,\pmb{\alpha}} \overbar I_h(\overbar J_k \times\mathcal {M}_{h,2}) (1+ o(1)) ,
\end{align*}
where the $(1)$-term is uniform in $1\leq k\leq m_h$ and $0<h\leq h_0$, because of assumption (B2). Noticing that $\overbar I_h(\overbar J_k \times\mathcal {M}_{h,2})=h^{-r_1}I_h(J_k \times\mathcal {M}_{h,2})$, we get the desired result.
\end{proof}

\begin{lemma}\label{lemma3.2}
For any $\epsilon>0$, there exist $\gamma_0>0$, $\theta_0>0$ such that for all $\gamma\leq\gamma_0$, $\theta\geq\theta_0$, $0<h\leq h_0$, and $J_k\in \{J_{k,h}, J_{k,h}^\delta, J_{k,h}^{-\delta}\}$ with $1\leq k\leq m_h$, we have for some $\epsilon_{k,h}$ with $|\epsilon_{k,h}|\leq \epsilon$,
\begin{align}\label{SumToInt2}
&\frac{\mathbb{Q}_{Z_h}(\theta, (J_k\times \mathcal {M}_{h,2})\cap \Gamma_{h,\gamma,\theta})}{\theta^{2(r_1/\alpha_1+r_2/\alpha_2)}\Psi(\theta)}
=(1+ \epsilon_{k,h})h^{-r_1}\;
\wt H_{R,\pmb{\alpha}}(\gamma) I_h(J_k \times\mathcal {M}_{h,2}),
\end{align}
where $\wt H_{R,\pmb{\alpha}}(\gamma)$ only depends on $\gamma$ such that $\wt H_{R,\pmb{\alpha}}(\gamma)\rightarrow H_{R,\pmb{\alpha}}$ as $\gamma\rightarrow0.$
%
\end{lemma}

%
%
%

\begin{proof}
The proof is similar to that of Lemma~\ref{lemma3.1}. The difference is that, instead of applying Theorem~\ref{fixedmanifold}, we use Lemma~\ref{fixedmanifolddis} in the appendix. Note that in order to apply Lemma~\ref{fixedmanifolddis}, we need to find a diffeomorphism $\psi_k$ from $J_k$ to $\mathbb{R}^r$, for each $k=1,\cdots,m_h$. This diffeomorphism is constructed in the same way as shown at the beginning of the proof of Theorem~\ref{fixedmanifold}.
\end{proof}


\begin{lemma}\label{lemma3.3}
%
%
For $\theta=\theta_{h,z}$ given in (\ref{ThetaExp}) with any fixed $z$, we have that  as $h \to 0,$ 
\begin{align}\label{PhiX}
h^{-r_1}\theta^{2(r_1/\alpha_1+r_2/\alpha_2)}\Psi(\theta)=\frac{e^{-z}}{H_{R,\pmb{\alpha}} I_h(\mathcal{M}_{h})}(1+o(1))=O(1). 
\end{align}
\end{lemma}

\begin{proof}
Observe that the first equality in (\ref{PhiX}) follows from a direct calculation using (\ref{ThetaExp}). 
Next we show (\ref{PhiX}) is bounded. Recall that $\|D_{\bm{t},h} P_{\bm{t}}\|_{r_1}= \det[(P_{\bm{t}}^T D_{\bm{t},h}^T D_{\bm{t},h} P_{\bm{t}})^{1/2}]$ (see (\ref{determine})), where the columns of $P_{\bm{t}}$ are orthonormal and span the tangent space $T_{\bm{t}}\mathcal{M}_h$. Since $D_{\bm{t},h}$ is non-singular, there exists an orthogonal matrix $E_{\bm{t},h}$ such that the columns of $P_{\bm{t}}$ are the eigenvectors of $E_{\bm{t},h}D_{\bm{t},h}$, whose associated eigenvalues are denoted by $\lambda_{\bm{t},1},\cdots,\lambda_{\bm{t},r_1}$. Let $\Lambda_{\bm{t}}=\text{diag}(\lambda_{\bm{t},1},\cdots,\lambda_{\bm{t},r_1})$. Then
\begin{align*}
\|D_{\bm{t},h} P_{\bm{t}}\|_{r_1}= & \det[(P_{\bm{t}}^T D_{\bm{t},h}^T E_{\bm{t},h}^T E_{\bm{t},h} D_{\bm{t},h} P_{\bm{t}})^{1/2}] \\
= & \det[(\Lambda_{\bm{t}} P_{\bm{t}}^T  P_{\bm{t}} \Lambda_{\bm{t}})^{1/2}] \\
=&\prod_{j=1}^{r_1} |\lambda_{\bm{t},j}|.
\end{align*}

The above calculation also shows that $\lambda_{\bm{t},1}^2,\cdots,\lambda_{\bm{t},r_1}^2$ are eigenvalues of $D_{\bm{t},h}^T E_{\bm{t},h}^T E_{\bm{t},h} D_{\bm{t},h}=D_{\bm{t},h}^TD_{\bm{t},h}$. It then follows that
\begin{align*}
[\lambda_{\min}(D_{\bm{t},h}^TD_{\bm{t},h})]^{r_1/2} \leq \|D_{\bm{t},h} P_{\bm{t}}\|_{r_1} \leq [\lambda_{\max}(D_{\bm{t},h}^TD_{\bm{t},h})]^{r_1/2}.
\end{align*}

The left-hand side in (\ref{PhiX}) is bounded because with assumption (B2) we have
\begin{align*}
0 < & \inf_{0<h\leq h_0, \bm{t}\in\mathcal{M}_h} [\lambda_{\min}(D_{\bm{t},h}^T D_{\bm{t},h})]^{r_1/2} \inf_{0<h\leq h_0} \mathcal{H}_{r_1}(\mathcal{M}_h) \\
\leq & \inf_{0<h\leq h_0}I_h(\mathcal{M}_{h}) \leq \sup_{0<h\leq h_0}I_h(\mathcal{M}_{h}) \\
\leq & \sup_{0<h\leq h_0, \bm{t}\in\mathcal{M}_h} [\lambda_{\max}(D_{\bm{t},h}^T D_{\bm{t},h})]^{r_1/2}\sup_{0<h\leq h_0} \mathcal{H}_{r_1}(\mathcal{M}_h) <\infty.
\end{align*}
\end{proof}

Denote $\mathcal{J}_{h}^\delta = \bigcup_{k\leq m_h}J_{k,h}^\delta$. Recall that $\mathcal{M}_h=\mathcal {M}_{h,1}\times \mathcal {M}_{h,2}$. Approximating $\mathcal{M}_h$ by $\mathcal{J}_{h}^\delta \times \mathcal {M}_{h,2}$  leads to the approximation of $\mathbb{Q}_{Z_h}(\theta,\mathcal{M}_h)$ by  $\mathbb{Q}_{Z_h}(\theta,\mathcal{J}_{h}^\delta \times \mathcal {M}_{h,2})$. The volume of $\bigcup_{k\leq m_h}J_{k,h}^{-\delta},$ i.e., the difference between the volumes of $\mathcal{M}$ and $\mathcal{J}_{h}^\delta$, is of the order $O(\delta)$ uniformly in $h$. As the next lemma shows, the order of the difference $\mathbb{Q}_{Z_h}(\theta,\mathcal{M}_h) - \mathbb{Q}_{Z_h}(\theta,\mathcal{J}_{h}^\delta \times \mathcal {M}_{h,2})$ turns out to be of the same order. 

\begin{lemma}\label{lemma3.4}
With $\theta=\theta_{h,z}$ given in (\ref{ThetaExp}), there exists $0<C<\infty$ such that for $\delta$ and $h$ small enough,
\begin{align}\label{G1}
0< \mathbb{P}_{Z_h}(\theta,\mathcal{J}_{h}^\delta \times \mathcal {M}_{h,2})  - \mathbb{P}_{Z_h}(\theta,\mathcal{M}_{h}) \leq C\delta,
\end{align}
and
\begin{align}\label{G1T}
0 < \sum_{k=1}^{m_h}\mathbb{Q}_{Z_h}(\theta, J_{k,h}\times\mathcal {M}_{h,2}) - \sum_{k=1}^{m_h}\mathbb{Q}_{Z_h}(\theta,  J_{k,h}^\delta \times \mathcal {M}_{h,2} )\leq C\delta.
\end{align}
\end{lemma}

\begin{proof}
Using (\ref{eigenbounds}), we have that 
\begin{align}\label{DMbound}
\sup_{0<h\leq h_0,\bm{t}\in\mathcal{M}_h} \|D_{\bm{t},h} M_{\bm{t}} \|_r \leq  C_1:= \sup_{0<h\leq h_0, \bm{t}\in\mathcal{M}_h} [\lambda_{\max}(D_{\bm{t},h}^T D_{\bm{t},h})]^{r_1/2} <\infty.
\end{align}

Also note that for all $h\in(0,h_0]$, there exists $0<C_2<\infty$ such that $\max_{1\leq k \leq m_h}\mathcal{H}_{r}(J_{k,h}^{-\delta} \times \mathcal{M}_h) \leq C_2 \delta h^{r_1}$. Our construction of the partition of the $\mathcal{M}_h$ guarantees that there exists $0<C_3<\infty$ such that $m_h\leq C_3h^{-r_1}$. Therefore
\begin{align}\label{sumbound}
\sum_{k=1}^{m_h}I_h(J_{k,h}^{-\delta} \times \mathcal {M}_{h,2}) \leq m_h \sup_{0<h\leq h_0,\bm{t}\in\mathcal{M}_h} \|D_{\bm{t},h} M_{\bm{t}} \|_r \max_{1\leq k \leq m_h}\mathcal{H}_{r}(J_{k,h}^{-\delta} \times \mathcal{M}_h) \leq C_1C_2C_3 \delta. 
\end{align}
Using Lemma~\ref{lemma3.1}, for any $\epsilon>0$, we have for $h$ small enough that

\begin{align*}
0&\leq \mathbb{Q}_{Z_h}(\theta_{h,z},\mathcal{M}_{h})-\mathbb{Q}_{Z_h}(\theta_{h,z},\mathcal{J}_{h}^\delta \times \mathcal {M}_{h,2})\\
%
%
&\leq \sum_{k=1}^{m_h}\mathbb{Q}_{Z_h}(\theta_{h,z}, J_{k,h}^{-\delta}\times\mathcal {M}_{h,2})\\
&\leq (1+\epsilon)h^{-r_1}\,H_{R,\pmb{\alpha}} \theta_{h,z}^{2(r_1/\alpha_1+r_2/\alpha_2)}\Psi(\theta) \sum_{k=1}^{m_h}I_h(J_{k,h}^{-\delta} \times \mathcal {M}_{h,2}).
\end{align*}
Then (\ref{G1T}) follows from Lemma~\ref{lemma3.3} and (\ref{sumbound}). Also (\ref{G1}) holds because
\begin{align*}
0< \mathbb{P}_{Z_h}(\theta,\mathcal{J}_{h}^\delta \times \mathcal {M}_{h,2})  - \mathbb{P}_{Z_h}(\theta,\mathcal{M}_{h}) \leq \sum_{k=1}^{m_h}\mathbb{Q}_{Z_h}(\theta_{h,z}, J_{k,h}^{-\delta}\times\mathcal {M}_{h,2}).
\end{align*}
%
%
%
%
%
\end{proof}

With $\Gamma_{h,\gamma,\theta}$ given in (\ref{fullgrid}), $(\mathcal{J}_{h}^\delta \times \mathcal {M}_{h,2})\cap\Gamma_{h,\gamma,\theta}$ is a grid over $\mathcal{J}_{h}^\delta \times \mathcal {M}_{h,2}$. Next we show that excursion probabilities over these two sets are close, by choosing both $h$ and the grid size to be sufficiently small. 
 
\begin{lemma}\label{lemma3.5}
With $\theta=\theta_{h,z}$ given in (\ref{ThetaExp}), we have that 
\begin{align}\label{G2}
\mathbb{P}_{Z_h}(\theta,  \mathcal{J}_{h}^\delta \times \mathcal {M}_{h,2}) = \mathbb{P}_{Z_h}(\theta,  (\mathcal{J}_{h}^\delta \times \mathcal {M}_{h,2}) \cap\Gamma_{h,\gamma,\theta})+ o(1)
\end{align}
and
\begin{align}\label{G2T}
\sum_{k=1}^{m_h}\mathbb{Q}_{Z_h}(\theta,  J_{k,h}^\delta \times \mathcal {M}_{h,2})=\sum_{k=1}^{m_h}\mathbb{Q}_{Z_h}(\theta,  (J_{k,h}^\delta \times \mathcal {M}_{h,2}) \cap\Gamma_{h,\gamma,\theta})+ o(1),
\end{align}
as $\gamma,h\to0.$\\
\end{lemma}
 
 \begin{proof}
Lemmas~\ref{lemma3.1} and \ref{lemma3.2} imply that for any $\epsilon>0$, there exist $\gamma_0>0$ and $\theta_0>0$ such that for all $\gamma\leq\gamma_0$ and $\theta\geq\theta_0$,
\begin{align*}
0&\leq \mathbb{Q}_{Z_h}(\theta,J_{k,h}^\delta\times \mathcal {M}_{h,2})-\mathbb{Q}_{Z_h}(\theta, (J_{k,h}^\delta \times \mathcal {M}_{h,2}) \cap\Gamma_{h,\gamma,\theta} )\\
&\leq \sum_{j=1}^{n_h} \sum_{i=1}^{N_h}\Big[\mathbb{Q}_{Z_h}(\theta,S_i^h \times U_j)-\mathbb{Q}_{Z_h}(\theta,(S_i^h \times U_j)\cap\Gamma_{h,\gamma,\theta})\Big]\\
&\leq \epsilon h^{-r_1}\;\theta^{2(r_1/\alpha_1+r_2/\alpha_2)}\Psi(\theta)H_{R,\pmb{\alpha}} I_h(J_{k,h}^\delta \times \mathcal {M}_{h,2}).
\end{align*}
%
%
%
As a result, 
{\allowdisplaybreaks
\begin{align*}
0&\leq \mathbb{Q}_{Z_h}(\theta,  \mathcal{J}_{h}^\delta \times \mathcal {M}_{h,2})-\mathbb{Q}_{Z_h}(\theta,  (\mathcal{J}_{h}^\delta \times \mathcal {M}_{h,2}) \cap\Gamma_{h,\gamma,\theta})\nonumber\\
&\leq\sum_{k=1}^{m_h}\Big[\mathbb{Q}_{Z_h}(\theta,  J_{k,h}^\delta \times \mathcal {M}_{h,2})-\mathbb{Q}_{Z_h}(\theta,  (J_{k,h}^\delta \times \mathcal {M}_{h,2}) \cap\Gamma_{h,\gamma,\theta})\Big]\nonumber\\
%
%
%
&\leq\epsilon h^{-r_1}\,\theta^{2(r_1/\alpha_1+r_2/\alpha_2)}\Psi(\theta)H_{R,\pmb{\alpha}}  I_h\big(\mathcal{J}_{h}^\delta \times \mathcal {M}_{h,2} \big)\nonumber\\
&\leq\epsilon h^{-r_1}\,\theta^{2(r_1/\alpha_1+r_2/\alpha_2)}\Psi(\theta)H_{R,\pmb{\alpha}}  I_h(\mathcal{M}_{h}).
\end{align*}
}
Then (\ref{G2}) and (\ref{G2T}) immediately follows from (\ref{PhiX}). 
%
%
%
%
\end{proof}

Recall that $(\mathcal{J}_{h}^\delta \times \mathcal {M}_{h,2})\cap \Gamma_{h,\gamma,\theta}$ gives a set of dense grid points in $\mathcal{J}_{h}^\delta \times \mathcal {M}_{h,2}$. 
For any $1\leq k\leq m_h$, denote the set $T_k^{h,\gamma,\theta}=(J_{k,h}^\delta \times \mathcal {M}_{h,2}) \cap\Gamma_{h,\gamma,\theta}.$ Define a probability measure $\wt{\mathbb{P}}$ such that under $\wt{\mathbb{P}}$ the vectors $(Z_h(\bm{t}):\; \bm{t}\in T_k^{h,\gamma,\theta})$ and $(Z_h(\bm{t}^\prime):\; \bm{t}^\prime\in T_{k^\prime}^{h,\gamma,\theta})$ are independent for $k\neq k^\prime$. In other words, $\wt{\mathbb{P}}_{Z_h}(\theta,  (\mathcal{J}_{h}^\delta \times \mathcal {M}_{h,2}) \cap\Gamma_{h,\gamma,\theta}) = \prod_{k\leq m_h} \mathbb{P}_{Z_h}(\theta,  (J_{k,h}^\delta \times \mathcal {M}_{h,2}) \cap\Gamma_{h,\gamma,\theta}).$ As the next lemma shows, the probability $\mathbb{P}_{Z_h}(\theta,  (\mathcal{J}_{h}^\delta \times \mathcal {M}_{h,2}) \cap\Gamma_{h,\gamma,\theta})$ can be approximated by using the probability measure $\wt{\mathbb{P}}$, if $\delta$ and $\gamma$ are small.\\[5pt] 

\begin{lemma}\label{lemma3.6}
For $\delta>0$ fixed and small enough, there exists $\gamma=\gamma(h)\rightarrow0$ as $h \to 0,$ such that with $\theta=\theta_{h,z}$ given in (\ref{ThetaExp}), we have
\begin{align}\label{G3}
\mathbb{P}_{Z_h}(\theta,  (\mathcal{J}_{h}^\delta \times \mathcal {M}_{h,2}) \cap\Gamma_{h,\gamma,\theta}) =\prod_{k\leq m_h} \mathbb{P}_{Z_h}(\theta,  (J_{k,h}^\delta \times \mathcal {M}_{h,2}) \cap\Gamma_{h,\gamma,\theta})+o(1).
\end{align}
\end{lemma}

\begin{proof}
Denote $\bm{t}=(\bm{t}_{(1)}^T,\bm{t}_{(2)}^T)^T$ and $\bm{t}^\prime=(\bm{t}_{(1)}^{\prime T},\bm{t}_{(2)}^{\prime T})^T$, where $\bm{t}_{(1)},\bm{t}_{(1)}^\prime\in\mathbb{R}^{n_1}$ and $\bm{t}_{(2)},\bm{t}_{(2)}^\prime\in\mathbb{R}^{n_2}$. For $\bm{t}\in T_k^{h,\gamma,\theta}$ and $\bm{t}^\prime\in T_{k^\prime}^{h,\gamma,\theta}$ with $k\neq k^\prime$, we have $\bm{t}_{(1)}\in J_{k,h}^\delta$ and $\bm{t}_{(1)}^\prime\in J_{k^\prime,m_h}^\delta,$ and hence for all $0<h\leq h_0$, we have 
\begin{align*}
\|\xi_h^{-1}(\bm{t})-\xi_h^{-1}(\bm{t}^\prime)\| \geq \|(\bm{t}_{(1)}-\bm{t}_{(1)}^\prime)/h\|\geq (2h\delta)/h = 2\delta >0.
\end{align*}
%
Let $r_h(\bm{t}_1, \bm{t}_2)$ be the covariance between $Z_h(\bm{t}_1)$ and $Z_h(\bm{t}_2)$, for $\bm{t}_1,\bm{t}_2\in\mathcal{M}_h$. Then assumption (B3) implies that that there exists $\eta =\eta(\delta)> 0,$ such that
\begin{align}\label{RijBound}
\sup_{0<h\leq h_0}\sup_{k\neq k^\prime}\sup_{\bm{t}\in T_k^{h,\gamma,\theta}} \sup_{\bm{t}^\prime\in T_{k^\prime}^{h,\gamma,\theta}}|r_h(\bm{t},\bm{t}^\prime)|<\eta < 1.
\end{align}
%
%
By Lemma 4.1 of Berman \cite{Berman:1971} (aslo see Lemma A4 of Bickel and Rosenblatt \cite{Bickel:1973a}), we have 
\begin{align}\label{TripleSum}
&\big|\mathbb{P}_{Z_h}(\theta,  (\mathcal{J}_{h}^\delta \times \mathcal {M}_{h,2}) \cap\Gamma_{h,\gamma,\theta}) - \wt{\mathbb{P}}_{Z_h}(\theta,  (\mathcal{J}_{h}^\delta \times \mathcal {M}_{h,2}) \cap\Gamma_{h,\gamma,\theta})\big| \nonumber\\
%
%
\leq&8\sum\limits_{1\leq k\neq k^\prime\leq m_h}\sum_{\bm{t}\in T_k^{h,\gamma,\theta}}\sum_{\bm{t}^\prime\in T_{k^\prime}^{h,\gamma,\theta}} \int_0^{|r_h(\bm{t},\bm{t}^\prime)|}\frac{1}{2\pi(1-\lambda^2)^{1/2}}\exp{\bigg(-\frac{\theta^2}{1+\lambda}\bigg)}d\lambda\nonumber\\
%
%
\leq &\sum\limits_{1\leq k\neq k^\prime\leq m_h}\sum_{\bm{t}\in T_k^{h,\gamma,\theta}}\sum_{\bm{t}^\prime\in T_{k^\prime}^{h,\gamma,\theta}} \zeta_h(\bm{t},\bm{t}^\prime),
\end{align}
where 
\begin{align*}
\zeta_h(\bm{t},\bm{t}^\prime) = \frac{4|r_h(\bm{t},\bm{t}^\prime)|}{\pi(1-\eta^2)^{1/2}}\exp{\bigg(-\frac{\theta^2}{1+|r_h(\bm{t},\bm{t}^\prime)|}\bigg)}.
\end{align*}
%
%
%
We take $\gamma = [v(h^{-1})]^{(1/(3r_1+3r_2))}$. Let $\omega$ be such that $0<\omega<\frac{2}{(1+\eta)}-1,$ and define
\begin{align*}
\mathcal{G}_{h,\gamma,\theta}^{(1)}&=\{(\bm{t},\bm{t}^\prime)\in T_k^{h,\gamma,\theta} \times T_{k^\prime}^{h,\gamma,\theta}:\; \|\bm{t}_{(1)}-\bm{t}_{(1)}^\prime\| < h(N_{h,\delta}^{(1)})^{\omega/r_1}\gamma \theta^{-2/\alpha_1}, 1\leq k\neq k^\prime\leq m_h\},\\
\mathcal{G}_{h,\gamma,\theta}^{(2)}&=\{(\bm{t},\bm{t}^\prime)\in T_k^{h,\gamma,\theta} \times T_{k^\prime}^{h,\gamma,\theta}:\; \|\bm{t}_{(1)}-\bm{t}_{(1)}^\prime\|\geq h(N_{h,\delta}^{(1)})^{\omega/r_1}\gamma \theta^{-2/\alpha_1}, 1\leq k\neq k^\prime\leq m_h\},
\end{align*}
%
where $N_{h,\delta}^{(1)}$ is given in (\ref{nhdeltastar}). Then the triple sum on the right-hand side of (\ref{TripleSum}) can be written as
\begin{align}\label{splitup}
\sum_{(\bm{t},\bm{t}^\prime)\in \mathcal{G}_{h,\gamma,\theta}^{(1)}} \zeta_h(\bm{t},\bm{t}^\prime) + \sum_{(\bm{t},\bm{t}^\prime)\in \mathcal{G}_{h,\gamma,\theta}^{(2)}} \zeta_h(\bm{t},\bm{t}^\prime).
\end{align}
Note that the cardinality of $\mathcal{G}_{h,\gamma,\theta}^{(1)}$ is of the order $O((N_{h,\delta}^{(1)})^{\omega+1} (N_{h}^{(2)})^2)$, where $N_{h}^{(2)}$ is given in (\ref{nustar}). Hence for the first sum in (\ref{splitup}) we have
\begin{align}\label{part1sum}
\sum_{(\bm{t},\bm{t}^\prime)\in \mathcal{G}_{h,\gamma,\theta}^{(1)}} \zeta_h(\bm{t},\bm{t}^\prime)=&O\bigg((N_{h,\delta}^{(1)})^{\omega+1}(N_{h}^{(2)})^2\exp\bigg\{-\frac{\theta^2}{1+\eta}\bigg\}\bigg) \nonumber\\
=&O\bigg(\bigg(\frac{\theta^{2r_1/\alpha_1}}{h^{r_1}\gamma^{r_1}}\bigg)^{1+\omega}    \frac{\theta^{4r_2/\alpha_2}}{\gamma^{2r_2}}   \exp\bigg\{-\frac{\theta^2}{1+\eta}\bigg\}\bigg) \nonumber\\
=&O\bigg(\bigg(\frac{(\log\frac{1}{h})^{r_1/\alpha_1 +2r_2/[\alpha_2(1+\omega)]}}{h^{r_1}\gamma^{r_1+2r_2/(1+\omega)}}\bigg)^{1+\omega}\exp\bigg\{-\frac{2r_1\log\tfrac{1}{h}}{1+\eta}\bigg\}\bigg) \nonumber\\
=&O\bigg(h^{\frac{2r_1}{1+\eta}-r_1(1+\omega)}\Big(\log\tfrac{1}{h}\Big)^{\frac{(1+\omega) r_1}{\alpha_1} + \frac{2r_2}{\alpha_2}}\Big(v(\tfrac{1}{h})\Big)^{-\frac{(1+\omega)r_1+2r_2}{3r_1}}\bigg) \nonumber\\
=& o(1) \quad \textrm{as} \quad h\rightarrow 0.
\end{align}

Now we consider the second sum in (\ref{splitup}). Due to (\ref{RijBound}) and $(1+|r_h(\bm{t},\bm{t}^\prime)|)^{-1}\geq 1-|r_h(\bm{t},\bm{t}^\prime)|$, we have
\begin{align*}
\zeta_h(\bm{t},\bm{t}^\prime) \leq \frac{4|r_h(\bm{t},\bm{t}^\prime)|}{\pi(1-\eta^2)^{1/2}}\exp{\big(-(1-|r_h(\bm{t},\bm{t}^\prime)|)\theta^2\big)}.
\end{align*}
%
%
Since $\theta^2=O(\log{\frac{1}{h}})$ and $\exp(-\theta^2)=O(h^{-2r_1})$, we have $\exp{\big(-(1-|r_h(\bm{t},\bm{t}^\prime)|)\theta^2\big)} = O(h^{-2r_1})$ for $(\bm{t},\bm{t}^\prime)\in \mathcal{G}_{h,\gamma,\theta}^{(2)}$ by using (\ref{SupGauss2}). 
Hence 
%
when $h$ is sufficiently small, there exists a constant $C>0$ such that 
\begin{align}
\sup_{(\bm{t},\bm{t}^\prime)\in \mathcal{G}_{h,\gamma,\theta}^{(2)}} \zeta_h(\bm{t},\bm{t}^\prime) \leq C h^{2r_1} \frac{v((N_{h,\delta}^{(1)})^{\omega/r_1}\gamma \theta^{-2/\alpha_1})}{[\log((N_{h,\delta}^{(1)})^{\omega/r_1}\gamma \theta^{-2/\alpha_1})]^{2r_1/\alpha_1+2r_2/\alpha_2}}.
\end{align}
Therefore it follows from (\ref{SupGauss2}) that 
\begin{align}\label{part2sum}
\sum_{(\bm{t},\bm{t}^\prime)\in \mathcal{G}_{h,\gamma,\theta}^{(2)}} \zeta_h(\bm{t},\bm{t}^\prime) &= O\left(h^{2r_1}(N_{h,\delta}^{(1)})^2 (N_{h}^{(2)})^2\frac{v((N_{h,\delta}^{(1)})^{\omega/r_1}\gamma \theta^{-2/\alpha_1})}{[\log((N_{h,\delta}^{(1)})^{\omega/r_1}\gamma \theta^{-2/\alpha_1})]^{2r_1/\alpha_1+2r_2/\alpha_2}}\right) \nonumber\\
&=O\left(\frac{(\log\frac{1}{h})^{2r_1/\alpha_1+2r_2/\alpha_2}v((N_{h,\delta}^{(1)})^{\omega/r_1}\gamma \theta^{-2/\alpha_1})}{\bigg[\log \bigg(h^{-\omega} \Big((\log \frac{1}{h})^{1/\alpha_1}v(\frac{1}{h})^{-1/3r_1}\Big)^{\omega-1}\bigg)\bigg]^{2r_1/\alpha_1+2r_2/\alpha_2}\big(v(\frac{1}{h})\big)^{2/3}}\right) \nonumber\\
&=o(1) \quad \textrm{as} \quad h\rightarrow 0.
\end{align}
Combining (\ref{TripleSum}), (\ref{part1sum}) and (\ref{part2sum}), we obtain (\ref{G3}).
%
%
\end{proof}

{\bf Proof of Theorem~\ref{ProbMain}}
\begin{proof}
We choose the same $\gamma=\gamma(h)$ in Lemma~\ref{lemma3.6}, and use $\theta=\theta_{h,z}$ given in (\ref{ThetaExp}). Fix a small $\delta>0$. By using (\ref{G1}), (\ref{G2}), and (\ref{G3}), we have that as $h\rightarrow0$,
\begin{align*}
\mathbb{P}_{Z_h}(\theta, \mathcal{M}_{h}) & = \prod_{k\leq m_h} \mathbb{P}_{Z_h}(\theta,  (J_{k,h}^\delta \times \mathcal {M}_{h,2}) \cap\Gamma_{h,\gamma,\theta})+o(1)\\
&= \exp\bigg\{\sum_{k\leq m_h}\log{\Big(1-\mathbb{Q}_{Z_h}(\theta,  (J_{k,h}^\delta \times \mathcal {M}_{h,2}) \cap\Gamma_{h,\gamma,\theta})\Big)}\bigg\}+o(1)\\
&=\exp\bigg\{-(1+ o(1))\sum_{k\leq m_h}\mathbb{Q}_{Z_h}(\theta,  (J_{k,h}^\delta \times \mathcal {M}_{h,2}) \cap\Gamma_{h,\gamma,\theta})\bigg\}+ o(1).
\end{align*}
Then by using (\ref{G2T}), (\ref{G1T}), and (\ref{SumToInt}), we get
\begin{align*}
\mathbb{P}_{Z_h}(\theta, \mathcal{M}_{h}) & = \exp\big\{-(1+o(1))h^{-r_1} \;\theta^{2(r_1/\alpha_1+r_2/\alpha_2)}\Psi(\theta)H_{R,\pmb{\alpha}} I_h(\mathcal{M}_{h} ) \big\}+o(1).
\end{align*}
The proof is completed by noticing (\ref{PhiX}).
\end{proof}

\section{Appendix}\label{appendix}
In this appendix, we collect some miscellaneous results that are straightforward extensions from some existing results in the literature, and have been used in our proofs.\\

For an integer $\ell>0$ and $\gamma>0$, let $C(\ell,\gamma) =\{t\gamma: \; t \in [0, \ell ]^n  \cap {\mathbb Z}^n \}$. Given a structure $(E,\pmb{\alpha})$, let $H_{E,\pmb{\alpha}}(\ell,\gamma) = H_{E,\pmb{\alpha}}(C(\ell,\gamma)) $ and 
\begin{align*}
H_{E,\pmb{\alpha}}(\gamma) = {\displaystyle \lim_{\ell\rightarrow\infty}}\frac{H_{E,\pmb{\alpha}}(\ell,\gamma)}{\ell^n}.
\end{align*} 
%
%
%
%
The existence of this limit follows from Pickands \cite{Pickands:1969b}. 
Using the factorization lemma (Lemma 6.4 of Piterbarg \cite{Piterbarg:1996}) and Theorem B3 of Bickel and Rosenblatt \cite{Bickel:1973b}, we have 
\begin{lemma}\label{HGamma}
$H_{E,\pmb{\alpha}}=\lim_{\gamma\rightarrow0}\frac{H_{E,\pmb{\alpha}}(\gamma)}{\gamma^n}$.
\end{lemma}

Let $\Gamma_{E,\pmb{\alpha}}(\gamma,u) = \{(x_1,\cdots,x_k)\in\mathbb{R}^n: \; x_i = \gamma u^{-2/\alpha_i}{\mathbf{\ell}_i}, \mathbf{\ell}_i\in\mathbb{Z}^{e_i}, i=1,\cdots,k\}.$ 
%
The following result extends Lemma 4.2 in Qiao and Polonik \cite{Qiao:2018} from assuming a simple structure with $E=\{n\}$ and a scalar $0<\pmb{\alpha}\leq2$ to a more general structure. The proof uses similar ideas and therefore is omitted. Also see Lemma 3 of Bickel and Rosenblatt \cite{Bickel:1973b}, and Lemma 7.1 of Piterbarg \cite{Piterbarg:1996}. 
\begin{lemma}\label{Piece}
Given a structure $(E,\pmb{\alpha})$, let $X(\bm{t})$, $\bm{t}\in\mathbb{R}^{n}$, be a centered homogeneous Gaussian field with covariance function $r(\bm{t})=\mathbb{E}(X(\bm{t}+\bm{s})X(\bm{s}))=1-|\bm{t}|_{E,\pmb{\alpha}}(1+(1)),$ as $\bm{t}\rightarrow0$. Then there exists $\delta_0>0$ such that for any closed Jordan measurable set $A$ of positive $n$-dimensional Lebesgue measure with diameter not exceeding $\delta_0$, the following asymptotic behavior occurs: 
%
\begin{align*}
\mathbb{P}\left( \sup_{\bm{t}\in A_{\gamma,u}} X(\bm{t}) >u\right) = \frac{H_{E,\pmb{\alpha}}(\gamma)}{\gamma^n}\mathcal{H}_n(A) \prod_{i=1}^k u^{2e_i/\alpha_i} \Psi(u) (1+o(1)),
\end{align*}
as $u\rightarrow\infty$, where $A_{\gamma,u}=A\cap \Gamma_{E,\pmb{\alpha}}(\gamma,u)$.
%
%
%
%
%
\end{lemma}

The next theorem is similar to Theorem 7.1 of Piterbarg \cite{Piterbarg:1996}, except that the supremum is over a dense grid. The proof is similar, where one need to replace the role of Lemma 7.1 of Piterbarg \cite{Piterbarg:1996} by our Lemma~\ref{Piece} above. 
\begin{theorem} \label{Theorem7.1ana}
Let $X(\bm{t})$, $\bm{t}\in A\subset\mathbb{R}^n$ be a locally-$(E,\pmb{\alpha},D_{\bm{t}})$-stationary Gaussian field with zero mean, where $A$ is a closed Jordan set of positive $n$-dimensional Lebesgue measure. Assume also that the matrix-valued function $D_{\bm{t}}$ is continuous in $\bm{t}$ and non-singular everywhere on $A$. Then if $r_X(\bm{t},\bm{s})<1$ for all $\bm{t},\bm{s}$ from $A$, $\bm{t}\neq\bm{s}$, the following asymptotic behavior occurs:
\begin{align*}
\mathbb{P}\left( \sup_{\bm{t}\in A_{\gamma,u}} X(\bm{t}) >u\right) = \frac{H_{E,\pmb{\alpha}}(\gamma)}{\gamma^n}\int_A |\det D_{\bm{t}}| d\bm{t} \prod_{i=1}^k u^{2e_i/\alpha_i} \Psi(u) (1+o(1)),
\end{align*}
as $u\rightarrow\infty$, where $A_{\gamma,u}=A\cap \Gamma_{E,\pmb{\alpha}}(\gamma,u)$.
\end{theorem}

The following lemma is analogous to Lemma~\ref{fixedanxillary} with the index set being a grid. The proof is also similar to that of Lemma~\ref{fixedanxillary}, except that in the proof we use Theorem~\ref{Theorem7.1ana} to replace the role of Theorem 7.1 of Piterbarg \cite{Piterbarg:1996}.
\begin{lemma}\label{fixedmanifolddis}
Suppose that the conditions in Theorem~\ref{fixedmanifold} hold. For any subset $U\subset \mathcal{M}$, if there exists a diffeomorphism $\psi:U\mapsto \Omega\subset\mathbb{R}^r$, where $\Omega=\psi(U)$ is a closed Jordan set of positive $r$-dimensional Lebesgue measure, then we have that as $u\rightarrow\infty$, 
\begin{align}\label{subsetextreme2}
\mathbb{P} \left(\sup_{\bm{t}\in M_{\gamma,u}} X(\bm{t}) > u\right) = \frac{H_{R,\pmb{\alpha}}(\gamma)}{\gamma^r} \int_{\mathcal{M}} \int_{U} \prod_{j=1}^k \|D_{j,\bm{t}} P_{j,\bm{t}}\|_{r_j} d\mathcal{H}_r(\bm{t})  \prod_{i=1}^k u^{2r_i/\alpha_i} \Psi(u) (1+o(1)),
\end{align}
where $M_{\gamma,u}=\psi^{-1}(\Omega\cap \Gamma_{R,\pmb{\alpha}}(\gamma,u))$.
\end{lemma}

\end{section}


\begin{thebibliography}{99}
%
\bibitem{Adler:2007} {\sc Adler, R.J. and Taylor, J.E.} (2007). \emph{Random Fields and Geometry}, Springer, New York.
%
\bibitem{Albin:2016} {\sc Albin, J.M.P., Hashorva, E.,  Ji, L. and Ling, C.} (2016). Extremes and limit theorems for difference of chi-type processes. {\em ESAIM Probab. Stat}, {\bf 20}, 349-366.
%
\bibitem{Azais:2009} {\sc Aza\"{i}s, J.-M. and Wschebor, M.} (2009). \emph{Level Sets and Extrema of Random Processes and Fields}, John Wiley \& Sons, Hoboken, NJ.
%
\bibitem{Bai:2018} {\sc Bai, L.} (2018). Extremes of locally-stationary chi-square processes on discrete grids. {\em ArXiv: 1807.11687}.
%
%
\bibitem{Berman:1964} {\sc Berman, S.M.} (1964). Limit theorems for the maximum term in stationary sequences. {\em Ann. Math. Statist.}, {\bf 35}, 502-516.
%
\bibitem{Berman:1971} {\sc Berman, S.M.} (1971). Asymptotic independence of the numbers of high and low level crossings of stationary Gaussian processes. {\em Ann. Math. Statist.}, {\bf 42}, 927--945.
%
%
%
%
%
\bibitem{Bickel:1973a} {\sc Bickel, P. and Rosenblatt, M.} (1973a). On some global measures of the deviations of density function estimates. {\em Ann. Statist.}, {\bf 1}, 1071--1095.
%
\bibitem{Bickel:1973b} {\sc Bickel, P. and Rosenblatt, M.} (1973b). Two-dimensional random fields, in {\em Multivariate Analysis III, P.K. Krishnaiah, Ed.} pp. 3--15, Academic Press, New York.
%
\bibitem{Boissonnat:2018} {\sc Boissonnat, J.-D., Chazal, F. and Yvinec, M.} (2018). {\em Geometric and Topological Inference}. Cambridge University Press, New York, NY. 
%
\bibitem{Borell:1975} {\sc Borell, C.} (1975). The Brunn-Minkowski inequality in Gauss space. {\em Invent. Math.}, {\bf 30}, 207--216.
%
\bibitem{Broida:1989} {\sc Broida, J.G. and Willamson, S.G.} (1989): {\em A Comprehensive Introduction to Linear Algebra.} Addison-Wesley.
%
%
%
%
%
\bibitem{Cheng:2017} {\sc Cheng, D.} (2017). Excursion probabilities of isotropic and locally isotropic Gaussian random fields on manifolds. {\em Extremes}, {\bf 20}, 475-487.
%
\bibitem{Cheng:2016} {\sc Cheng, D. and Xiao, Y.} (2016). Excursion probability of Gaussian random fields on sphere. {\em Bernoulli}, {\bf 22}, 1113-1130.
%
\bibitem{Chernozhukov:2014} {\sc Chernozhukov, V., Chetverikov, D. and Kato, K.} (2014). Gaussian approximation of suprema of empirical processes. {\em Ann. Statist.}, {\bf 42}, 1564--1597.
%
%
%
%
%
\bibitem{Cuevas:2012} {\sc Cuevas, A., Fraiman, R., and Pateiro-L\'{o}pez, B.} (2012). On statistical properties of sets fulfilling rolling-type conditions. {\em Advances in Applied Probability} {\bf 44} 311-329.
%
%
\bibitem{Evans:1992} {\sc Evans, L.C. and Gariepy, R.F.} (1992). {\em Measure Theory and Fine Properties of Functions}. CRC Press, Boca Raton, FL.
%
\bibitem{Federer:1959} {\sc Federer, H.} (1959). Curvature measures. {\em Trans. Amer. Math. Soc.}, {\bf 93}, 418--491.
%
%
%
%
%
\bibitem{Genton:2015} {\sc Genton, M.G. and Kleiber, W.} (2015). Cross-covariance functions for multivariate geostatistics. {\em Statistical Science}, {\bf 30}, 147--163.
%
%
%
%
\bibitem{Hashorva:2015} {\sc Hashorva, E. and Ji, L.} (2015). Piterbarg theorems for chi-processes with trend {\em Extremes}, {\bf 18}, 37-64.
%
%
%
%
%
%
\bibitem{Ji:2019} {\sc Ji, L. Liu, P. and Robert, S.} (2019). Tail asymptotic behavior of the supremum of a class of chi-square processes. {\em Statistics \& Probability Letters}. {\bf 154}, 108551.
%
%
\bibitem{Konakov:1984} {\sc Konakov, V.D., and Piterbarg, V.I.} (1984). On the convergence rate of maximal deviations distributions for kernel regression estimates. {\em J. Multivariate Anal.}, {\bf 15}, 279--294.
%
\bibitem{Konstantinides:2004} {\sc Konstantinides D., Piterbarg V. and Stamatovic S.} (2004). Gnedenko-type limit theorems for cyclostationary $\chi^2$-processes. {\em Lith. Math. J.}, {\bf 44}(2), 157-167.
%
%
%
%
\bibitem{Lindgren:1989} {\sc Lindgren, G.} (1989). Slepian models for $\chi^2$-processes with dependent components with application to envelope upcrossings. {\em J. Appl. Probab.}, {\bf 26} (1), 36-49. 
%
\bibitem{Ling:2016} {\sc Ling, C. and Tan, Z.} (2016). On maxima of chi-processes over threshold dependent grids. {\em Statistics}, {\bf 50}(3), 579-595.
%
\bibitem{Liu:2016} {\sc Liu, P. and Ji, L.} (2016). Extremes of chi-square processes with trend. {\em Probab. Math. Statist.}, {\bf 36}(1).
%
\bibitem{Liu:2017} {\sc Liu, P. and Ji, L.} (2017). Extremes of locally stationary chi-square processes with trend. {\em Stochastic Process. Appl.} {\bf 127}(2), 497-525.
%
\bibitem{Mikhaleva:1997}  {\sc Mikhaleva, T.L. and Piterbarg, V.I.} (1997). On the distribution of the maximum of a Gaussian field with constant variance on a smooth manifold. {\em Theory Probab. Appl.}, {\bf 41}, 367--379.
%
\bibitem{Niyogi:2008}  {\sc Niyogi, P., Smale, S. and Weinberger, S.} (2008). Finding the homology of submanifolds with high confidence from random samples. {\em Discrete and Computational Geometry} {\bf 39} 419--441.
%
%
\bibitem{Pickands:1969b} {\sc Pickands, J. III.} (1969b). Upcrossing probabilities for stationary Gaussian processes. {\em Trans. Amer. Math. Soc.}, {\bf 145}, 51--73.
%
\bibitem{Piterbarg:1994} {\sc Piterbarg, V.I.} (1994). High excursion for nonstationary generalized chi-square processes. \emph{Stochastic Processes and their Applications}, {\bf 53} 307-337.
%
\bibitem{Piterbarg:1996} {\sc Piterbarg, V.I.} (1996). \emph{Asymptotic Methods in the Theory of Gaussian Processes and Fields}, Translations of Mathematical Monographs, Vol. 148, American Mathematical Society, Providence, RI.
%
\bibitem{Piterbarg:2001} {\sc Piterbarg, V.I. and Stamatovich, S.} (2001). On maximum of Gaussian non-centered fields indexed on smooth manifolds. In {\it Asymptotic Methods in Probability and Statistics with Applications; Statistics for Industry and Technology,}\,  Eds: N. Balakrishnan, I. A. Ibragimov, V. B. Nevzorov, Birkh\"{a}user, Boston, MA, pp. 189--203.
%
\bibitem{Qiao:2019} {\sc Qiao, W.} (2020). Asymptotic confidence regions for density ridges, {\em arXiv: 2004.11354}.
%
\bibitem{Qiao:2018} {\sc Qiao, W. and Polonik, W.} (2018). Extrema of rescaled locally stationary Gaussian fields on manifolds, {\em Bernoulli}, {\bf 24}(3), 1834-1859.
%
\bibitem{Qiao:2019b} {\sc Qiao, W. and Polonik, W.} (2019). Nonparametric confidence regions for level sets: statistical properties and geometry. {\em Electronic Journal of Statistics}, {\bf 13}(1), 985-1030.
%
%
\bibitem{Rosenblatt:1976} {\sc Rosenblatt, M.} (1976). On the maximal deviation of $k$-dimensional density estimates. {\em Ann. Probab.}, {\bf 4}, 1009--1015.
%
\bibitem{Scholtes:2013} {\sc Scholtes, S.} (2013). On hypersurfaces of positive reach, alternating Steiner formulae and Hadwiger's Problem. {\em arXiv:1304.4179}.
%
%
%
%
%
%
\bibitem{Tan:2013a} {\sc Tan, Z. and Hashorva, E.} (2013a). Exact asymptotics and limit theorems for supremum of stationary $\chi$-processes over a random interval. {\em Stochastic Processes and their Applications}. {\em 123}(8), 2983-2998.
%
\bibitem{Tan:2013b} {\sc Tan, Z. and Hashorva, E.} (2013b). Limit theorems for extremes of strongly dependent cyclo-stationary $\chi$-processes. {\em Extremes}. {\em 16}(2), 241-254.
%
%
\bibitem{Tan:2014} {\sc Tan, Z. and Wu, C.} (2014). Limit laws for the maxima of stationary chi-processes under random index. {\em TEST}. {\bf 23}(4). 769-786.
%
%
\bibitem{Zhou:2017} {\sc Zhou, Y. and Xiao, Y.} (2017). Tail asymptotics for the extremes of bivariate Gaussian random fields. {\em Bernoulli}. {\bf 23}(3). 1566-1598.

\end{thebibliography}

\end{document}